\documentclass[a4paper]{article}

\usepackage[utf8]{inputenc}
\usepackage[T1]{fontenc}
\usepackage{indentfirst}
\usepackage[a4paper,top=2.54cm,bottom=2.0cm,left=3.0cm,right=3.54cm]{geometry}

\usepackage[round,comma]{natbib}

\usepackage{hyperref}
\usepackage{float}

\usepackage{amsmath}
\usepackage{amssymb}
\usepackage{amsthm}

\usepackage{tikz}
\usetikzlibrary{arrows.meta}
\tikzset{
    circ/.style={draw, circle,inner sep=0pt,minimum size=8mm, font=\scriptsize},
    triangle/.tip={Computer Modern Rightarrow[open,angle=120:3pt]}
}
\usepackage{pgfplots}

\newtheorem{theorem}{Theorem}
\newtheorem{corollary}[theorem]{Corollary}
\newtheorem{lemma}[theorem]{Lemma}
\newtheorem{proposition}[theorem]{Proposition}
\newtheorem{remark}[theorem]{Remark}

\def\P{{\mathbb P}}

\newcommand{\A}{\mathcal{A}}

\newcommand{\st}{\mathcal{S}}

\newcommand{\pcons}{\mathcal{C}^+}
\newcommand{\ncons}{\mathcal{C}^-}
\newcommand{\ladder}{\mathcal{L}}
\newcommand{\pladder}{\mathcal{L}^+}
\newcommand{\nladder}{\mathcal{L}^-}

\newcommand{\str}{\hat{\mathcal{S}}}
\newcommand{\strength}{\alpha}

\newcommand{\pladderr}{\hat{\mathcal{L}}^{+,\strength}}
\newcommand{\mydef}{:=}

\title{\textbf{\Large{Fast Consensus and Metastability \\  in a Highly Polarized
     Social Network}} \\ \vspace{0.5cm} \textbf{ {\normalsize Dedicated to the memory of Antonio Galves}}}

\author{ 
Antonio Galves\footnote{Deceased September 5, 2023.} \and \stepcounter{footnote} \stepcounter{footnote} \stepcounter{footnote} \stepcounter{footnote} \stepcounter{footnote} 
Kádmo  Laxa \footnote{email: kadmo.laxa@usp.br}
}

\date{August, 2024}

\begin{document}

\maketitle

\begin{abstract}
A polarized social network is modeled as a system of interacting marked point processes with memory of variable length. Each point process indicates the successive times in which a social actor expresses a "favorable"  or "contrary" opinion. After expressing an opinion, the social pressure on the actor is reset to 0, waiting for the group’s reaction. The orientation and the rate at which an actor expresses an opinion is influenced by the social pressure exerted on it, modulated by a polarization coefficient. We prove that the network reaches an instantaneous but metastable consensus, when the polarization coefficient diverges.
 \vspace{0.2cm}
 
 
 \vspace{0.2cm}
 
\end{abstract}

\section{Introduction}
\label{cap:introducao}

Huge discrepancies between the results of electoral intentions carried out a few days before the actual voting and the electoral poll results during the first rounds of the 2018 and 2022 presidential elections in Brazil were striking
(see for instance \cite{exame}, \cite{bbc}, \cite{nyt} and \cite{guardian}).

It was conjectured that these discrepancies in the Brazilian elections were due to 
campaigns in support of a group of candidates launched on social media a few days before the elections (see for instance \cite{conversation} and \cite{mello1,mello2}).
This conjecture raises a question: is social media campaigning enough to change in a quite short period of time the voting intention of a significant portion of voters? 

To address this question, we introduce a new stochastic model that mimics some important features of real world social networks. This model can be informally described as follows. 
 
\begin{enumerate}
    \item The model
is a system with interacting marked point processes. 

\item Each point process indicates the successive times in which a social actor expresses either a ``favorable''  $(+1)$ or ``contrary'' $(-1)$  opinion on a certain topic.

\item When an actor expresses an opinion, social pressure on it is reset to 0, and simultaneously, social pressures on all the other actors change by one unit in the direction of the opinion that was just expressed.

\item The orientation and the rate at which an actor expresses an opinion is influenced by the social pressure exerted on it, modulated by a polarization coefficient. The greater the value of the polarization coefficient, the greater the speed in which actors express opinions and the greater the tendency of each actor to express opinions in the same direction of its social pressure.

\end{enumerate}

In this model, the actors listen to and influence each other. Thus, after expressing an opinion, each actor waits for the group's reaction to the opinion it has just expressed. It is this reaction that will guide the actor's next manifestation. This is the content of the third point presented above. The resetting of the social pressure on the actor every time it expresses an opinion makes the point processes describing the activity of the actor to have a memory of variable length.
In addition, the fact that each actor is always willing to accept the opinion of the other actors makes our social network model a kind of ``consensus building machine''. We will come back to this last point in the Discussion Section at the end of this article.

Starting with the classical voter model, introduced by \cite{votermodel}, several articles addressed issues associated to opinion dynamics in a social network. See \cite{wasserman1,socialdynamics} and \cite{aldous}  for a general review on this subject.
However, to the best of our knowledge, a model with the features introduced here was not considered yet in the literature of social networks.

The model introduced here belongs to the same class of systems of interacting point processes with memory of variable length that was introduced in discrete time 
by \cite{glmodel} 
and in continuous time by 
\cite{gl4} to model systems of spiking neurons. For a self-contained and neurobiological motivated presentation of this class of variable length memory models for systems of spiking neurons, both
in discrete and continuous time, we refer the reader to \cite{gb}.

Several questions have been studied for models in this class, including metastability \cite{mandre1,mandre2,taille,evamonm,mandre3,laxa}, phase transition \cite{lud,gl2, amarcos}, hydrodynamic limits and propagation of chaos \cite{gl4, ostduarte2,evafour}, 
time evolution and stationary states \cite{ostduarte1,lebow},
replica-mean-field limits \cite{baccelli1,gl3,baccelli2,baccelli3}, existence and perfect simulation \cite{glmodel, karina} and statistical inference of the graph of interaction between neurons \cite{ostduarte3,desantis}. However, none of these articles addressed the consensus issue considered here.

A main difference between the model considered in the present article and the models considered in the articles mentioned above is the fact that in our case we consider marked point processes, each mark indicating whether the expressed opinion is favorable or contrary on a certain topic. A second important difference lies in the presence of a polarization coefficient tuning the social influence on the opinion expressed by each actor.
These specific characteristics make it possible to study the constitution of consensus situations. Our model, results and proofs are new and original.

Let us now informally present our results. Our model is described by the time evolution of the list containing the social pressures of the actors. The existence of the process and the uniqueness of its invariant probability measure is the content of Theorem \ref{existence}.

When the polarization coefficient diverges,
the invariant probability measure concentrates on a subset of the set of consensus lists and the time the system needs to get there goes to zero. Here by a consensus list, we mean any list in which all the social pressures push in the same direction. This is the content of 
Theorem \ref{fastconsensus}, which gives a mathematical rigorous meaning of the informal description of a social network as a ``consensus building machine''.

In the social network, the consensus has a metastable behavior. This means that the direction of the social pressures on the actors globally changes after a long and unpredictable random time. This is the content of Theorem \ref{metastable}.
 
The notion of metastability considered here is inspired by the so called \textit{pathwise approach to metastability}
introduced by \cite{cassandro}. 
For more references and an introduction to the topic, we refer the reader to \cite{metaest2,metaest3} and \cite{metaest1}.

The original motivation of this article was to provide a mathematical framework to model global changes of the voting intention. The model with the four features informally described above accounts for the fast constitution of consensus in our model of social network. But something is missing in the model, namely the effect of an external influencer campaigning to push the global orientation of the social network in a certain direction. This can be done by adding the following fifth feature to the model.

\begin{itemize} 
\item[5.] A robot may be present in the system. A robot behaves as an actor who always expresses the same opinion with the same time rate. This time rate 
increases exponentially with the polarization coefficient and with a strictly
positive parameter describing the ``strength'' of the robot.  
\end{itemize}

For the model of social network with external influence, let us suppose that the strength of the robot is sufficiently big. Then, as the polarization coefficient diverges, the invariant probability measure of the system concentrates on a subset of the set of consensus lists oriented in the direction in which the robot campaigns. Moreover, the system reaches this subset very fast. This is the content of Theorem \ref{fastconsensus2}.

This article is organized as follows. In Section \ref{cap:model} we define the model of a social network without external influence and state the main results for this model. In Section \ref{cap:modelrobot} we define the model of a social network with a robot and state the theorem describing how the robot affects the behavior of the social network.
In Section \ref{sec:auxiliary} we introduce some extra notation and prove two auxiliary propositions. 
In Sections \ref{sec:existence}, \ref{sec: fastconsensus}, \ref{cap:metastability}, and \ref{sec:robot}, we prove Theorems \ref{existence}, \ref{fastconsensus}, \ref{metastable}, and \ref{fastconsensus2} respectively.

\section{Definitions, notation and main results}
\label{cap:model}

Let $\A=\{1,2,...,N\}$ be the set of social actors, with $N \geq 3$, and let $\mathcal{O}=\{-1,+1\}$ be the set of opinions that an actor can express, where $+1$ (respectively, $-1$) represents  a \textit{favorable} opinion (respectively, a \textit{contrary} opinion).

Let $\beta \geq 0$ be the polarization coefficient of this network. The polarization coefficient of the network parametrizes the tendency of each actor $a \in \A$ to follow the social pressure that the actors belonging to $\A \setminus \{a\}$ exert on $a$.

A list of social pressures $u=(u(a): a \in \A)$ is a list in which $u(a)$ is an integer number indicating the social pressure of actor $a \in \A$. To describe the time evolution of the social network we introduce a family of maps on the set of lists of social pressures. For any actor $a \in \A$, for any opinion $\mathfrak{o} \in \mathcal{O}$ and for any list of social pressures $u=(u(a): a \in \A)$, 
we define the new list $\pi^{a,\mathfrak{o}}(u)$ as follows. For all $b \in \A$, 
$$
\pi^{a,\mathfrak{o}}(u)(b)\mydef
\begin{cases}
u(b)+\mathfrak{o}, \text{ if } b\neq a, \\
0, \text{ if } b  = a.
\end{cases}
$$

The time evolution of the social network can be described as follows.
\begin{itemize}
    \item Assume that at time $0$, the list of social pressures exerted on the actors is $u=(u(a): a \in \A)$.
    \item Independent exponential random times with parameters $\exp\left(\beta \mathfrak{o} u(a)\right)$ are associated to each actor $a \in \A$ and each opinion $\mathfrak{o} \in \mathcal{O}$.

    \item Denote $(A_1,O_1)$ the pair (actor, opinion) associated to the exponential random time that occurs first.
    \item At this random time, the list of social pressures changes from $u$ to $\pi^{A_1,O_1}(u)$.
    
    \item At the new list of social pressures $\pi^{A_1,O_1}(u)$, independent exponential random times  with parameters $\exp\left(\beta \mathfrak{o} \pi^{A_1,O_1}(u)(a)\right)$  are associated to each actor $a \in \A$ and opinion $\mathfrak{o} \in \mathcal{O}$. 
    
    \item Denote $(A_2,O_2)$  the pair (actor, opinion) associated to the exponential random time that occurs first, and so on.
\end{itemize}

Note that for any $\beta >0$ and $u(a) \neq 0$, $\exp\left(\beta \mathfrak{o} u(a)\right)>1$ when the signs of $\mathfrak{o}$ and $ u(a)$ are equal and $\exp\left(\beta \mathfrak{o} u(a)\right)<1$ in the opposite case. When $\beta=0$, the social network is a system with independent components in which actors express either opinion $+1$ or $-1$ with equal probabilities (each actor waits an exponential random time with rate $2$ to express an opinion)
regardless of the social pressure exerted on them.  

Let $(T_n: n\geq 1)$ be the cumulative sums of the successive random times realizing the successive minima and let $((A_n,O_n):n\geq 1)$ be the sequence of pairs (actor, opinion) associated to them. For each $n\geq 1$, $T_n$ is the sum of $n$ exponentially distributed random variables.
Let also $(U_t^{\beta,u})_{t\in [0,+\infty)}$ be the time evolution of the list containing the social pressure of the actors,  starting with the list $u$, defined as follows 
$$
U_t^{\beta,u}\mydef u \text{, if } 0\leq t <T_1,
$$
and for any $t\geq T_1$,
$$
U_t^{\beta,u} \mydef \pi^{A_m,O_m}(U^{\beta,u}_{T_{m-1}}) \text{, if } T_{m}\leq t <T_{m+1},
$$
where $T_0=0$. 

So defined, the time evolution of the list of social pressures $(U_t^{\beta,u})_{t\in [0,+\infty)}$ is a Markov jump process taking values in the set
$$
\st \mydef \{u=(u(a): a \in \A) \in \mathbb{Z}^N:\min\{|u(a)|:a \in \A\}=0\}
$$
for any initial list $u \in \st$ and with infinitesimal generator defined as follows
\begin{equation} \label{generator}
\mathcal{G}f(u)\mydef \sum_{\mathfrak{o} \in \mathcal{O}}\sum_{b\in \A}\exp\left(\beta \mathfrak{o} u(b)\right)\left[f(\pi^{b,\mathfrak{o}}(u))-f(u)\right],
\end{equation}
for any  bounded function $f:\st \to \mathbb{R}$. Note that $U_t^{\beta,u} \in \st$ for any $t\geq 0$ and for any initial list $u \in \st$ since the actor that expressed an opinion most recently at each instant has null social pressure.

The opinion dynamics of the social network is described either by the system of interacting marked point processes $((T_n,(A_n,O_n)):n\geq 1)$ together with the initial list of social pressures $U_0^{\beta,u}=u$, or by the time evolution of the list of social pressures $(U_t^{\beta,u})_{t\in [0,+\infty)}$. 

The process $(U_t^{\beta,u})_{t\in [0,+\infty)}$ is well defined for any 
$t \in [0, \sup\{T_m:m\geq 1\})$. The unique thing that must yet be clarified is whether this process is defined for any positive time $t$, i.e. if $\sup\{T_m:m\geq 1\}=+\infty$ or not. This is part of the content of the first theorem.

\begin{theorem} \label{existence}
For any $\beta\geq 0$ and for any starting list $u \in \st$, the following holds.
\begin{enumerate}
    \item The sequence $(T_m: m\geq 1)$ of jumping times of the process $(U_t^{\beta,u})_{t \in [0,+\infty)}$ satisfies
    $$
    \P(\sup\{T_m: m\geq 1\}=+\infty)=1,
    $$ 
    which assures the existence of the process for all time $t\in [0,+\infty)$. 
    
    \item The process $(U_t^{\beta,u})_{t\in [0,+\infty)}$ 
    has a unique invariant probability measure $\mu^{\beta}$. 
    \end{enumerate}
    
\end{theorem}

We define the set of positive (respectively negative) consensus lists as
$$
\pcons\mydef\{u \in \st: u \neq \vec{0}, u(a) \geq 0 \text{, for all } a \in \A\}
$$
and 
$$
\ncons\mydef\{u \in \st: u \neq \vec{0}, u(a) \leq 0 \text{, for all } a \in \A\},
$$
where $\vec{0} \in \st$ is the null list. 
Consider also the set of positive and negative ladder lists
$$
\pladder\mydef\{u \in \st: \{u(1),...,u(N)\}=\{0,1,...,N-1\}\}
$$
and
$$
\nladder\mydef\{u \in \st: \{u(1),...,u(N)\}=\{0,-1,...,-(N-1)\}\}.
$$
The set of ladder lists is given by $\ladder\mydef\pladder \cup \nladder$.

To state the next theorems, we define for any $u \in \st$ and for any $B \subset \st$, the hitting time $R^{\beta,u}(B)$ as follows
$$
R^{\beta,u}(B)\mydef \inf\{t>0: U_t^{\beta,u} \in B\}.
$$
Theorem \ref{fastconsensus} states that the invariant measure gets concentrated in the set of the ladder lists as the polarization coefficient diverges. Moreover, for any non-null initial list, the time it takes for the process to reach the set of ladder lists goes to $0$, as the polarization coefficient diverges. 
Corollary \ref{coro: fastconsensus} in Section \ref{sec: fastconsensus} presents an equivalent result when the initial list is the null list.
In this case, after the first jump time, which is exponentially distributed, the time it takes for the process to reach the set of ladder lists goes to $0$, as the polarization coefficient diverges.

\begin{theorem} \label{fastconsensus} 
 \begin{enumerate}
 \item[]
\item There exists a constant $C=C(N)>0$, such that for any $\beta\geq 0$  the invariant probability measure $\mu^{\beta}$ satisfies
$$
\mu^{\beta}(\ladder) \geq  1-C e^{-\beta}.
$$

\item 
 For any fixed $\delta>0$ 
$$
\sup_{u \in \st \setminus\{\vec{0}\}}\P\Big(R^{\beta,u}(\ladder)>e^{-\beta(1-\delta)}\Big) \to 0 \text{, as } \beta \to +\infty.
$$
\end{enumerate}
\end{theorem}

\vspace{0.3cm}

Theorem \ref{metastable} states that a highly polarized social network has a metastable behavior.

\begin{theorem}\label{metastable} 
For any $v \in \pcons$,
$$
\frac{R^{\beta,v}(\ncons)}{\mathbb{E}[R^{\beta,v}(\ncons)]}\to \text{Exp}(1) \text{ in distribution, as } \beta \to +\infty,
$$
where $\text{Exp}(1)$ denotes the mean $1$ exponential distribution.
\end{theorem}

\section{The social network under the influence of a robot}
\label{cap:modelrobot}

Theorems \ref{existence}, \ref{fastconsensus} and \ref{metastable} describe the behavior of the social network without any external influence.
To model the effect of an external influencer campaigning to push the global orientation of the social network in a certain direction, we add an extra feature to the model, namely the presence of a ``robot'' in the social network. 
Informally speaking, a robot behaves as an actor who is not influenced by the other actors and
always expresses the same opinion.  To simplify the presentation, we assume that the robot will campaign for the ``favorable'' opinion. 

Formally speaking, to describe the effect of the robot campaigning for the ``favorable'' opinion we consider an extra map $\pi^{*,+1}$ defined as follows. For any list of social pressures $u=(u(a): a \in \A)$, 
the new list $\pi^{*,+1}(u)$ satisfies
$$
\pi^{*,+1}(u)(b)\mydef
u(b)+1, \text{ for all } b\in \A.
$$

Let  
$$
\hat{U}_t^{\beta,u}= (\hat{U}_t^{\beta,u}(a): a \in \A)
$$ 
be the list of social pressures on the actors of the social network under the influence of the robot at time $t\geq 0$. The time evolution of this list $(\hat{U}_t^{\beta,u})_{t\in [0,+\infty)}$ is a Markov jump process taking values in $\str\mydef\mathbb{Z}^N$ with infinitesimal generator defined as follows
\begin{equation} \label{generator2}
\hat{\mathcal{G}}f(u)\mydef \sum_{\mathfrak{o} \in \mathcal{O}}\sum_{b\in \A}\exp\left(\beta \mathfrak{o} u(b)\right)\left[f(\pi^{b,\mathfrak{o}}(u))-f(u)\right]+\exp\left(\beta  \strength\right)\left[f(\pi^{*,+1}(u))-f(u)\right],
\end{equation}
for any  bounded function $f:\str \to \mathbb{R}$. In the above formula $\strength >0$ is a strictly positive parameter describing the ``strength'' of the robot.

The robot campaign for the ``favorable'' opinion with rate $\text{exp}(\beta \alpha)$, regardless of the social pressure exerted on the actors. Each time the robot campaign for the ``favorable'' opinion the social pressures on all the actors increase by one unit. Note that the robot is not an element of $\A$.

For any $\beta\geq 0$, for any $\strength >0$ and for any initial list $u \in \str$, the process $(\hat{U}_t^{\beta,u})_{t \in [0,+\infty)}$ exists and has a unique invariant probability measure $\nu^{\beta, \strength}$. The proof of this statement is similar as the proof of Theorem \ref{existence}.

For any $\strength \geq N-1$, we define $\pladderr$ as the set of lists in which all actors have different social pressures and the social pressure of each actor is non negative and
smaller than or equal to $\lfloor \strength\rfloor +1$
$$
\pladderr \mydef\left\{u \in \str: \bigcap_{a\in \A}\bigcap_{b\neq a}\{u(a)\neq u(b)\} \text{ and } 0 \leq u(a) \leq \lfloor \strength\rfloor +1 \text{, for all }  a \in \A \right\}.
$$
Observe that the definition of $\pladderr$ only makes sense when $\strength \geq N-1$. Note also
that for any $u \in \pladderr$, there exists $l \in \pladder$ such that $u(a) \geq l(a)$ for all $a\in \A$. In the case of a social network with a robot campaigning for the favorable opinion,
the set  $\pladderr$ will play the same role as the set of ladder lists $\ladder$ in a social network without external influence. This is the content of the next Theorem \ref{fastconsensus2}.

Theorem \ref{fastconsensus2} describes the behavior of the system with a robot pushing in the $+1$ direction.
If the strength of the robot $\strength$ is greater than $N-1$, then, as the coefficient polarization diverges, the system reaches very fast the set $\pladderr$ in which all the social pressures are positive. Moreover, in this case the invariant measure of the system of social pressures gets concentrated exponentially fast in the set $\pladderr$ as $\beta \to +\infty$. 

To state Theorem \ref{fastconsensus2}, we define for any $u \in \str$ and for any $B \subset \str$, the hitting time $\hat{R}^{\beta,u}(B)$ as follows
$$
\hat{R}^{\beta,u}(B)\mydef \inf\{t>0: \hat{U}_t^{\beta,u} \in B\}.
$$

\begin{theorem} \label{fastconsensus2}  For the model of social network with a robot with time evolution described by \eqref{generator2},
if $\strength > N-1$, then the following holds.
 \begin{enumerate} 
\item There exists a constant $\hat{C}=\hat{C}(N)>0$, such that for any $\beta\geq 0$  the invariant probability measure $\nu^{\beta, \strength}$ of the system satisfies
$$
\nu^{\beta, \strength}(\pladderr) \geq  1-\hat{C} e^{-\beta}.
$$
\item For any fixed $\delta>0$,
$$
\sup_{u \in \str 
}\P\Big(\hat{R}^{\beta,u}(\pladderr)>e^{-\beta \alpha(1-\delta)}\Big) \to 0 \text{, as } \beta \to +\infty. 
$$
\end{enumerate}
\end{theorem}

\section{Auxiliary notation and results} \label{sec:auxiliary}

In this section we will prove some auxiliary results that will be used to prove Theorems \ref{existence}, \ref{fastconsensus} and \ref{metastable}. To do this, we need to extend the notation introduced before.

\vspace{0.2cm}

\noindent \textbf{Extra notation}
\begin{itemize}
    \item 
     The Markov chain embedded in the process  $(U_t^{\beta,u})_{t\in [0,+\infty)}$ will be denoted  $(\tilde{U}^{\beta,u}_n)_{n \geq 0}$. In other terms, 
 $$\tilde{U}^{\beta,u}_0=u \ \text{ and } \
 \tilde{U}^{\beta,u}_n=U_{T_n}^{\beta,u}, \text{ for any $n\geq 1$. }
 $$
 
 \item 
 $(\tilde{U}_n^{\beta, u})_{n\geq 0}$ is a positive-recurrent Markov chain (see the proof of Part 2 of Theorem \ref{existence}). Its invariant probability measure  will be denoted $\tilde{\mu}^{\beta}$.
 
 \item For any list $u \in \st$, the first return time of the embedded Markov chain $(\tilde{U}_n^{\beta, u})_{n\geq 0}$ to $u$ will be denoted
 $$\tilde{R}^{\beta,u}(u) \mydef \inf\{n\geq 1:\tilde{U}^{\beta,u}_n=u\}.$$

\item For any $u\in \st$, the opposite list $-u \in \st$ is given by
$$
(-u)(a)=-u(a) \text{, for all } a \in \A.
$$

\item Let $\sigma: \A \to \A$ be a bijective map. For any $u\in \st$, the permuted list $\sigma(u)\in \st$ is given by
$$
\sigma (u)(a)=u(\sigma(a))  \text{, for all } a \in \A.
$$

\item The following event will appear several times in what follows. For any initial list $u \in \st$ and for any $n \geq 1$, 
$$
M_n^u \mydef \{A_n \in \text{argmax}\{|\tilde{U}_{n-1}^{\beta,u}(a)|: a \in \A\} \text{\ \  and \ \ } O_n\tilde{U}_{n-1}^{\beta,u}(A_n) \geq 0\}.
$$
$M_n^u$ is the event in which the $n$-th opinion expressed in the process with initial list $u$ is expressed by one the actors with greatest social pressure in absolute value and in the same direction of its social pressure. Note that this is the most likely choice of the pair (actor, opinion).

\end{itemize}

In Part 1 of Proposition \ref{teo:m1m} we show that the probability of the event $M_n^u$ for all $n=1,\ldots, m$, i.e.,  the sequentially occurrence of the most likely choice of the pair (actor, opinion) in the process,
is lower bounded for any initial list and this lower bound approaches $1$ as $\beta \to +\infty$. In Part 2 of Proposition \ref{teo:m1m} we show that the occurrence of the event $M_n^u$ for all $n=1,\ldots, 3(N-1)$ implies that $\tilde{U}^{\beta,u}_{3(N-1)} \in \ladder$.  

\begin{proposition} \label{teo:m1m}
\begin{enumerate}
\item[]
\item For any $u \in \st$ and for any $m \geq 1$,
$$
\P\left(\bigcap_{j=1}^mM_j^u\right) \geq (\zeta_{\beta})^m,
$$
where
$$
\zeta_{\beta} \mydef \frac{e^{\beta}}{e^{\beta}+e^{-\beta}+2(N-1)}\to 1, \text{ as } \beta \to +\infty.
$$
\item 
For any $u \in \st$,
$$
\P\left(\tilde{U}^{\beta,u}_{3(N-1)} \in \ladder\ \Big| \bigcap_{j=1}^{3(N-1)}M_j^u \right)=1.
$$
\end{enumerate}
\end{proposition}

We will first prove Part $1$ of Proposition \ref{teo:m1m}.

\begin{proof}

To prove Part $1$ of Proposition \ref{teo:m1m}, we first observe that by the Markov property, 
 \begin{equation}
     \label{eq1prop11}
 \P\left(\bigcap_{j=1}^mM_j^u\right)=
 \sum_{v \in \st}\P\left(\bigcap_{j=1}^{m-1}M_j^u,\tilde{U}^{\beta,u}_{m-1}=v  \right)\P(M_1^v).
\end{equation}
To obtain a lower bound for \eqref{eq1prop11} we will obtain a lower bound for $\P(M_1^u)$. Note that in the case $u=\vec{0}$, the probability $\P(M_1^u)$ is maximized.
For any $u \in \st \setminus\{\vec{0}\}$,
$$
\P(M_1^u)=\frac{|Y(u)|e^{\beta y(u)}}{|Y(u)|(e^{\beta y(u)}+e^{-\beta y(u)})+ \displaystyle\sum_{b\not\in Y(u)}(e^{\beta u(b)}+e^{-\beta u(b)})},
$$
where
$
y(u)=\max\{|u(a)|: a\in \A\}
$
and
$
Y(u)=\{a \in \A: |u(a)|=y(u)\}.
$
By rearranging the terms of the equation above, we have that 
$$
\P(M_1^u)=\frac{1}{(1+e^{-2\beta y(u)})+ \displaystyle\frac{1}{|Y(u)|}\displaystyle\sum_{b\not\in Y(u)}(e^{\beta (u(b)-y(u))}+e^{-\beta (u(b)+y(u))})}.
$$
Since $|u(b)|-y(u) \leq -1$, for any $b \not\in Y(u)$, we have that
$(e^{\beta (u(b)-y(u))}+e^{-\beta (u(b)+y(u))}) \leq 2e^{-\beta}$, for any $b\not\in Y(u)$. This implies that,
$$
\P(M_1^u)\geq\frac{1}{(1+e^{-2\beta y(u)})+ \displaystyle\frac{N-|Y(u)|}{|Y(u)|}2e^{-\beta}}\geq \zeta_{\beta},
$$
where in the last inequality we use the fact that for any $u\neq \vec{0}$, $|Y(u)|\geq 1$ and $y(u) \geq 1$.
Applying this lower bound $m$ times in Equation \eqref{eq1prop11}, we conclude the proof of Part 1.

\end{proof}

The proof of Part 2 of Proposition \ref{teo:m1m} is based on two Lemmas. Before proving Lemmas \ref{lemma1m1m} and \ref{lemma2m1m}, let us introduce some extra notation.

Let $I_n \mydef \{1,\ldots, n\}$, for each $n\geq 1$.
For any $u\in \st$, let
$$
n^+(u)\mydef\inf\{n\geq 1: I_n \subset \{u(a):a\in \A\} \text{ and } \min\{u(a):a\in \A\} \geq -(n-1)\},
$$
$$
n^-(u)\mydef\inf\{n \geq 1: I_n \subset \{-u(a):a\in \A\} \text{ and } \max\{u(a):a\in \A\} \leq (n-1)\},
$$
$$
n(u)\mydef\max\{n^-(u),n^+(u)\},
$$
with the convention $\inf\{\emptyset\}=0$. Note that $\min\{n^+(u),n^-(u)\}=0$, by definition.

Let $\mathcal{S}^+ \mydef \{u \in \st: n^+(u) \geq 1\}$ and $\mathcal{S}^- \mydef \{u \in \st: n^-(u) \geq 1\}$. In other words, $\mathcal{S}^+$ is the set of lists $u \in \st$ such that there exists $n\in\{1,\ldots, N-1\}$ and a sequence of $n$ different actors $a_1(u),\ldots, a_{n}(u)$ satisfying
\begin{equation*} 
  u(a_j(u))=j, \text{ for } j=1,...,n \quad \text{ and } \quad 
u(a) \geq -(n-1),
 \text{ for all } a \in \A.
\end{equation*}
$n^+(u)$ is the smallest $n$ such that the above conditions are satisfied. $\mathcal{S}^-$ and $n^-(u)$ are described in an analogous way.

For any initial list $u \in \st$, let
$$
\tau(u) \mydef \inf\{n\geq1: A_n \in \{A_1,..., A_{n-1}\} \cup\{a \in \A: u(a)=0\}\}
$$ 
be the number of times in which an actor expressed an opinion in the process until the first instant in which either one actor with null initial social pressure expresses an opinion or an actor expresses an opinion for the second time. Since we have $N$ actors and at least one actor has null social pressure, it follows that $\tau(u) \leq N$.

To prove Part 2 of Proposition \ref{teo:m1m} we first prove that the occurrence of the event 
$M_{\tau(u)}^u$ implies that $\tilde{U}^{\beta,u}_{\tau(u)} \in \mathcal{S}^+ \cup  \mathcal{S}^-$. This is the content of Lemma \ref{lemma1m1m}. Furthermore, for any $u\in \mathcal{S}^+\cup \mathcal{S}^-$, the occurrence of the event $M_n^u$ for all $n=1,\ldots, n(u)-1$ implies that $\tilde{U}^{\beta,u}_{n(u)-1)} \in \pcons\cup \ncons$. This is the content of Lemma \ref{lemma2m1m}. Putting all this together we are able to prove Part 2 of Proposition \ref{teo:m1m}.

\begin{lemma} \label{lemma1m1m}

For any initial list $u \in \st$, we have 
$$
\P\left(\tilde{U}_{\tau(u)}^{\beta,u}\in \mathcal{S}^+ \cup \mathcal{S}^- \ \big| \ M_{\tau(u)}^u\right)=1.
$$

\end{lemma}
\begin{proof}
By definition, the event $\{\tau(u)=1\}$ implies that $A_1 \in \{a\in \A: u(a)=0\}$.
Therefore,
$\{\tau(u)=1\} \cap M_1^u$ is only possible if $u=\vec{0}$. In this case, we have that $\tilde{U}_1^{\beta,u} \in  \mathcal{S}^+ \cup \mathcal{S}^-$ and then, the lemma holds when  $\{\tau(u)=1\}$.

Assume that $\tau(u) \geq 2$. At instant $\tau(u)-1$, we have that $\tilde{U}_{\tau(u)-1}^{\beta,u}(A_{\tau(u)-1})=0$ and since $A_j \neq A_k$, for each $1\leq j < k \leq \tau(u)-1$, we have that for any $j=1,\ldots,\tau(u)-2$,
$$
\tilde{U}_{\tau(u)-1}^{\beta,u}(A_j)=O_{j+1}+\ldots+O_{\tau(u)-1},
$$
which implies that
\begin{equation} \label{eq:spdif1}
|\tilde{U}^{\beta,u}_{\tau(u)-1}(A_j)-\tilde{U}^{\beta,u}_{\tau(u)-1}(A_{j-1})|=|O_j|=1.
\end{equation}
Moreover, for any $a \in \A$ such that $u(a)=0$, by the definition of $\tau(u)$ we have that $a \not\in \{A_1,\ldots, A_{\tau(u)-1}\}$. 
This implies that
$$
\tilde{U}_{\tau(u)-1}^{\beta,u}(a)=O_1+\ldots+O_{\tau(u)-1},
$$
and then, 
\begin{equation} \label{eq:spdif2}
|\tilde{U}_{\tau(u)-1}^{\beta,u}(A_1)-\tilde{U}_{\tau(u)-1}^{\beta,u}(a)|=1.
\end{equation}
Let
$$
m \mydef \max\{|O_{\tau(u)-1}|,|O_{\tau(u)-2}+O_{\tau(u)-1}|,\ldots, |O_1+\ldots+O_{\tau(u)-1}|\}.
$$
It follows from \eqref{eq:spdif1} and \eqref{eq:spdif2} (see Figure \ref{fig:M3}) that there exists a set with $m+1$ actors 
$$
\{a_0,a_1,\ldots, a_m\} \subset \{A_1,\ldots,A_{\tau(u)-1}\} \cup \{a \in \A: u(a)=0\},
$$ 
such that either 
$$
\{\tilde{U}_{\tau(u)-1}^{\beta,u}(a_j):j=0,1,\ldots,m\}=\{0,1,\ldots, m\}
$$
or
$$
\{\tilde{U}_{\tau(u)-1}^{\beta,u}(a_j):j=0,1,\ldots,m\}=\{0,-1,\ldots, -m\}.
$$

\begin{figure}[h]
\centering
\begin{tikzpicture}[scale=0.8]
\begin{axis}[
symbolic x coords={0,$a_0$, $A_1$, $A_2$, $A_3$, $A_4$, $A_5$, $A_6$, $A_7$, $A_8$, $A_{\tau-1}$, $a$}, xtick=data, 
axis y line=center, 
axis x line=bottom,     
axis x line shift=-1, 
] 
\addplot+[only marks,color=black,dashed] coordinates {($a_0$,0) ($A_1$, 4) ($A_2$, 4) ($A_3$, -1) ($A_4$, -1) ($A_5$, 3) ($A_6$, 2) ($A_7$, 3) ($A_8$, -1) ($A_{\tau-1}$, -1) ($a$,2)}; 
\addplot+[color=white,dashed,line width=0.7mm,mark=none] coordinates { (0, 4) ($A_{\tau-1}$, 4) };
\end{axis} 
\end{tikzpicture}
\hfill
\begin{tikzpicture}[scale=0.8]
\begin{axis}[
symbolic x coords={0,$a_0$, $A_1$, $A_2$, $A_3$, $A_4$, $A_5$, $A_6$, $A_7$, $A_8$, $A_{\tau-1}$}, xtick=data, 
axis y line=center, 
axis x line=bottom,     
axis x line shift=-1, 
] 
\addplot+[only marks,color=black,dashed] coordinates {($a_0$, 3) ($A_1$, 2) ($A_2$, 3) ($A_3$, 4) ($A_4$, 3) ($A_5$, 2) ($A_6$, 1) ($A_7$, 0) ($A_8$, -1) ($A_{\tau-1}$, 0)};  
\addplot+[color=red,dashed,line width=0.7mm,mark=none] coordinates { (0, 4) ($A_{\tau-1}$, 4) };
\end{axis} 
\draw[red,dashed] (2.05, 4.55) circle (0.2);
\draw[red,dashed] (2.75,5.68) circle (0.2);
\draw[red,dashed] (4.11, 3.41) circle (0.2);
\draw[red,dashed] (4.795, 2.265) circle (0.2);
\draw[red,dashed] (5.464, 1.142) circle (0.2);
\end{tikzpicture}

\caption{An example showing that 
either $\{0,1,\ldots,m\}$  
or $\{0,-1,\ldots,-m\}$ is a subset of $\{\tilde{U}_{\tau(u)-1}^{\beta,u}(a): a=a_0,\ldots, A_{\tau-1}\}$, where
$m=\max\{|\tilde{U}_{\tau(u)-1}^{\beta,u}(a)|: a=a_0,\ldots, A_{\tau-1}\}$. The figure in the left shows the initial list of social pressures $u$. In this example $N=11$, $\tau=10$, $\{a_0\} = \{b\in \A : u(b)=0\}$ and $\{a\} = \{b\in \A : u(b)\neq 0 \text{ and } b\neq A_j, j=1,\ldots, \tau(u)-1\}$.
The figure in the right shows $\tilde{U}_{\tau(u)-1}^{\beta,u}(a)$, for $a=a_0, A_1,\ldots, A_{\tau-1}$. In this example, $(O_1,\ldots, O_{\tau(u)-1})=(+1,-1,-1,+1,+1,+1,+1,+1,-1)$ and $m=4$. The red line represents this maximum social pressure. By the definition of $\tau$, it follows that $a_0,\ldots, A_{\tau-1}$ are different actors. As a consequence of  \eqref{eq:spdif1} and \eqref{eq:spdif2}, the absolute value of the difference between the social pressures of the subsequent actors on the figure in the right is $1$. In this example, this implies that $\{0,1,\ldots,m\} \subset \{\tilde{U}_{\tau(u)-1}^{\beta,u}(a): a=a_0,\ldots, A_{\tau-1}\}$. In fact, $\{A_7, A_6, A_5, A_2, A_3\}=\{0,1,\ldots,4\}$. This is represented by the red dashed circles. Note that the figure in the right remains the same if we change $u$ but keep the same values of $(O_1,\ldots, O_{\tau(u)-1})$ and $\tau(u)$.
}
\label{fig:M3}
\end{figure}

By assumption, 
$$
A_{\tau(u)} \in \text{argmax}\{|\tilde{U}_{\tau(u)-1}^{\beta,u}(a)|: a \in \A\} \text{\ \  and \ \ } O_{\tau(u)}\tilde{U}_{\tau(u)-1}^{\beta,u}(A_{\tau(u)}) \geq 0.
$$
This together with the definition of $\tau(u)$ implies that
$$
|\tilde{U}_{\tau(u)-1}^{\beta,u}(A_{\tau(u)})|=m \geq  |\tilde{U}_{\tau(u)-1}^{\beta,u}(a)| \text{, for all } a \in \A,
$$
and therefore, either 
$$
\{\tilde{U}_{\tau(u)}^{\beta,u}(a_j):j=0,1,\ldots,m-1\}=\{1,\ldots, m\} \  \text{ and } \ \tilde{U}_{\tau(u)}^{\beta,u}(a) \geq -(m-1), \text{ for all } a \in \A,
$$
or
$$
\{\tilde{U}_{\tau(u)}^{\beta,u}(a_j):j=0,1,\ldots,m-1\}=\{-1,\ldots, -m\}  \  \text{ and } \ \tilde{U}_{\tau(u)}^{\beta,u}(a) \leq m-1, \text{ for all } a \in \A.
$$
This concludes the proof of Lemma \ref{lemma1m1m}.

\end{proof}

\begin{lemma} \label{lemma2m1m} 
For any $u \in \mathcal{S}^+,$ we have
$$
\P\left(\tilde{U}_{n(u)-1}^{\beta,u} \in \pcons\ \Big| \ \displaystyle\bigcap_{j=1}^{n(u)-1}M_j^u\right)=1.
$$

\end{lemma} 
\begin{proof}

Note that $ \{u\in \mathcal{S}^+: n(u)=1\} \subset \pcons$. This implies that the lemma holds when $n(u)=1$.

Assume that $u \in \mathcal{S}^+$ and $n(u) \geq 2$.
By the definition of $\mathcal{S}^+$, if $M_1^u$ occurs, then  
$$
\tilde{U}_0^{\beta,u}(A_1)= \max\{|u(a)|: a \in \A\} \text{ and } O_1=+1.
$$
Moreover, for any $j=0,\ldots,n(u)-1$ there exists an actor $a_j(u) \in \A$ such that $\tilde{U}_0^{\beta,u}(a_j(u))=u(a_j(u))
=j$. Therefore,
$$
\tilde{U}_1^{\beta,u}(a_j(u))=u(a_j(u))+1=j+1.
$$
As a consequence, 
$$
\{1,\ldots,n(u)\} \subset \{\tilde{U}_1^{\beta,u}(a): a \in \A\} \quad \text{ and } \quad 
\tilde{U}_1^{\beta,u}(a) \geq -(n(u)-2),
 \text{ for all } a \in \A.
$$

In general, for any $k=1,\ldots, n(u)-1$, if $\bigcap_{j=1}^{k}M_j^u$ occurs, then
there exists a sequence of actors $a_0(\tilde{U}^{\beta,u}_{k-1})$,  $\ldots$,  $a_{n(u)}(\tilde{U}^{\beta,u}_{k-1})$ such that for any $j=0,\ldots, n(u)-1$, we have that
$$
\tilde{U}_k^{\beta,u}(a_j(\tilde{U}^{\beta,u}_{k-1}))=j+1,
$$
and therefore,
$$
\{1,\ldots,n(u)\} \subset \{\tilde{U}_k^{\beta,u}(a): a \in \A\} \ \text{ and } \ 
\tilde{U}_k^{\beta,u}(a) \geq -(n(u)-(k+1)),
 \text{ for all } a \in \A.
$$

We conclude the proof of Lemma \ref{lemma2m1m} by taking $k=n(u)-1$. 
\end{proof}

\begin{proof}
To prove Part 2 of Proposition \ref{teo:m1m}, we first 
observe that for any $u \in \st$, $\tau(u) \leq N$ and for any $u\in \mathcal{S}^+ \cup \mathcal{S}^-$,  $n(u) \leq N-1$, by definition.
Therefore, the Markov property implies the following inequality
$$
\P\left(\tilde{U}^{\beta,u}_{3(N-1)} \in \ladder\ \Big| \bigcap_{j=1}^{3(N-1)}M_j^u \right)\geq \P\left(\tilde{U}^{\beta,u}_{
\tau(u)} \in \mathcal{S}^+ \cup \mathcal{S}^-\ \Big| \ \bigcap_{j=1}^{\tau(u)}M_j^u \right)\times
$$
$$
\P\left(\tilde{U}^{\beta,u}_{\tau(u)+n(u)-1}\in \pcons \cup \ncons \ \Big| \ \tilde{U}^{\beta,u}_{
\tau(u)} \in \mathcal{S}^+ \cup \mathcal{S}^-\ , \bigcap_{j=\tau(u)+1}^{\tau(u)+n(u)-1}M_j^u \right)\times 
$$
$$
\P\left(\tilde{U}^{\beta,u}_{3(N-1)}\in \ladder \ \Big| \ \tilde{U}^{\beta,u}_{\tau(u)+n(u)-1}\in \pcons \cup \ncons\ , \ 
\bigcap_{j=\tau(u)+n(u)}^{3(N-1)}M_j^u \right).
$$
By Lemma \ref{lemma1m1m}, Lemma \ref{lemma2m1m} and the symmetric properties of the process, the two first terms of the right-hand side of the above equation are equal to $1$. 

Note that, by definition, for any $m\geq (N-1)$ and for any $v\in \pcons$,
$$
\P\left(\tilde{U}_m^{\beta,v} \in \pladder \ \Big| \ \bigcap_{j=1}^m M_j^v \right)=1.
$$
Since $\tau(u)+n(u)-1 \leq 2(N-1)$, this together with the symmetric properties of the process implies that
$$
\P\left(\tilde{U}^{\beta,u}_{3(N-1)}\in \ladder \ \Big| \ \tilde{U}^{\beta,v}_{\tau(u)+n(u)-1}\in \pcons \cup \ncons\ , \ 
\bigcap_{j=\tau(u)+n(u)}^{3(N-1)}M_j^u \right)=1.
$$
This concludes the proof of Part 2 of Proposition \ref{teo:m1m}.
\end{proof}

\begin{corollary} \label{coro:return}
For any $u\not\in \ladder$ and for any $m\geq 1$,
$$
\P\left(\tilde{U}^{\beta,u}_{m} =u \ \Big| \ \bigcap_{j=1}^{m}M_j^u \right)=0.
$$
\end{corollary}
\begin{proof}
It follows directly by Part 2 of Proposition \ref{teo:m1m}, by contradiction.
\end{proof}

For any fixed $l \in \pladder$, let $c_{\beta,l}$ be the positive real number such that
\begin{equation}\label{cbeta}
\P(
R^{\beta,l}(\nladder)>c_{\beta,l})=e^{-1}.
\end{equation}
Due to the symmetric properties of the process, it is clear that $c_{\beta,l}=c_{\beta,l'}$, for any pair of lists $l$ and $l'$ belonging to $\pladder$. Therefore, in what follows we will omit to indicate $l$ in the notation of $c_{\beta}.$
The next proposition gives a lower bound to $c_{\beta}$.

\begin{proposition} \label{cbetabound}
There exists  $C_1>0$ such that for any $\beta \geq 0$,
$$
c_{\beta} \geq C_1e^{\beta}.
$$

\end{proposition}

\begin{proof}
For any fixed $l \in \pladder$, let
$$
\tau_-^{(1)} \mydef \inf\{T_n: O_n=-1, U_{T_{n-1}}^{\beta, l}(A_{n}) > 0\}
$$ 
be the first time in which an actor with positive social pressure expresses an opinion  $-1$.
Consider also
$$
\tau_-^{(2)} \mydef \inf\{T_n :E_n^-\cap (E_{n-1}^- \cup E_{n-2}^-) \},
$$ 
where $E_n^-\mydef \{O_n=-1, U_{T_{n-1}}^{\beta, l}(A_{n})=0\}$, for $n\geq 1$.
$\tau_-^{(2)}$ is the first time in which two opinions $-1$ are expressed in the network by actors with null social pressure with at most one positive opinion expressed between them. We define for convention $E_0^-=\emptyset$.

Note that, starting from $l \in \pladder$, there are three ways to exit the set of positive consensus lists:
\begin{itemize}
    \item an actor with positive social pressure expresses an opinion $-1$;
    \item two opinions $-1$ are expressed in the network sequentially by actors with null social pressure;

    \item the alternation of opinions $+1$ and $-1$, which can lead the process to the null list.
\end{itemize}
Therefore,
\begin{equation} \label{eq:cbetalemma0}
R^{\beta,l}(\st \setminus \pcons)\geq\min\{\tau_-^{(1)},\tau_-^{(2)}\}.
\end{equation}

The rate in which the process has an actor with positive social pressure expressing the opinion $-1$ is bounded above by $(N-1)e^{-\beta}$. This implies that,
for any $t>0$,
$$
\P(\tau_-^{(1)} > t) \geq \P(\text{Exp}((N-1)e^{-\beta}) >t),
$$
where $\text{Exp}((N-1)e^{-\beta})$ is a random variable exponentially distributed with mean $e^{\beta}/(N-1).$  

Let $n_0^-\mydef 0$ and for $j\geq 1$,
$$
n_j^-\mydef\inf\{n> n_{j-1}^-: O_n=-1, U_{T_{n-1}}^{\beta,l}(A_n)=0\}.
$$
We have that
$$
\tau^{(2)}_-=\sum_{j=1}^{J}(T_{n_j^-}-T_{n_{j-1}^-}),
$$
where $J\mydef \inf\{j\geq 1: n_{j}^-=n_{j-1}^-+1 \text{ or } n_{j}^-=n_{j-1}^-+2\}.$

The rate in which the process has an actor with null social pressure expressing the opinion $-1$ is bounded above by $N$. This implies that, for any $j\geq 1$ and for any $t>0$,
$$
\P(T_{n_{j}^-}-T_{n_{j-1}^-}>t) \geq \P(E_j> t),
$$
where $(E_j)_{j\geq 1}$ is a sequence of i.i.d. random variables exponentially distributed with mean $N$. Moreover,
\begin{equation} \label{eq:cbetal1}
\P(\{O_{n+2}=-1\}\cup \{O_{n+1}=-1\} \
| \ \tilde{U}_{n}^{\beta,l}\in \pcons) \leq 2\times \frac{(N-1)}{(N-1)+e^{\beta}}.
\end{equation}
Note that if $(X_n)_{n\geq 1}$ is a sequence of random variables assuming values in $\{0,1\}$ with $\P(X_{1}=1) \leq p \in (0,1)$ and $\P(X_{n+1}=1|X_n= \ldots = X_1=0) \leq p$, for $n\geq 1$, then for any $k\geq 1$,
$$
\P(\inf\{n\geq 1: X_n=1\} \geq k) = \prod_{n=1}^{k-1}\P\left(X_n=0 \ \Big| \ \bigcap_{j=1}^{n-1}\{X_j=0\} \right) \geq (1-p)^{k-1}. 
$$
Putting this inequality together with \eqref{eq:cbetalemma0} and \eqref{eq:cbetal1}, we conclude that
$$
\P(J \geq k \ | \ \tau_-^{(1)} > T_{n_{k-1}^-})=\P\left(\bigcap_{j=1}^{k-1} \left\{ n_j^- \not \in \{n_{j-1}^-+1, n_{j-1}^-+2\} \right\} \ \Big| \ \tau_-^{(1)} > T_{n_k^-} \right)
$$
$$
\geq 
(1-\lambda_{\beta})^{k-1},
$$
where
$$
\lambda_{\beta}\mydef 2\times \frac{(N-1)}{(N-1)+e^{\beta}}.
$$


Therefore, for any $t>0$,
$$
\P(\tau_-^{(2)}>t |\tau_-^{(1)}> t ) \geq \P\left(\sum_{j=1}^GE_j> t\right),
$$
where $G$ is random variable independent from $(E_j)_{j\geq 1}$ with Geometric distribution assuming values in $\{1,2,...\}$ with parameter $\lambda_{\beta}$.
This implies that for any $t>0$,
$$
\P(\tau_-^{(2)}>t |\tau_-^{(1)}> t ) \geq \P\left(\text{Exp}(N\times \lambda_{\beta})> t\right),
$$
where $\text{Exp}(N\times \lambda_{\beta})$ is a random variable exponentially distributed with mean $1/(N\times \lambda_{\beta})$.

Therefore, for any $t>0$,
$$
\P(R^{\beta,l}(\st \setminus \pcons)>t) \geq 
\P[ \text{Exp}((N-1)e^{-\beta})>t]\P[ \text{Exp}(N\times \lambda_{\beta})>t]. 
$$
This implies that
\begin{align*}
e^{-1}=\P(R^{\beta,l}(\nladder)>c_{\beta}) \geq \P(R^{\beta,l}(\st \setminus \pcons)>c_{\beta}) \geq e^{-c_\beta\left((N-1)e^{-\beta} +
N\times \lambda_{\beta}\right)} ,
\end{align*}
and therefore
$$
c_{\beta}\geq \left((N-1)e^{-\beta} + N\times \lambda_{\beta}\right)^{-1}.
$$
With this we concluded the proof of Proposition \ref{cbetabound}.

\end{proof}

\section{Proof of Theorem \ref{existence}} \label{sec:existence}

To prove Part $1$ of Theorem \ref{existence} we study the first time in which the process has an opinion being expressed by an actor with social pressure in absolute value smaller than $N$ (recall that $N$ is the number of actors in the system). We prove
that starting from any initial list $u \in \st$, this occurs in the first $N$ jump times of the process. This is the content of Lemma \ref{inftmenor}. With this Lemma we are able to prove Part 1 of Theorem \ref{existence}.



\begin{lemma} \label{inftmenor} For any list $u \in \st$, 
$$
\inf\left\{n\geq 1: |U^{\beta,u}_{T_{n-1}}(A_{n})|< N\right\} \leq N.
$$
\end{lemma}
\begin{proof}

The initial list of social pressures $u$ belongs to $\st$. Therefore, there exists $a_0 \in \A$ such that $u(a_0)=0$.
By definition, for any $m\geq 1$,
$$
|U^{\beta,u}_{T_{m}}(a_0)|\leq \left|\sum_{j=1}^mO_{j}\right| \leq \sum_{j=1}^m\left|O_{j}\right|=m.
$$
More generally,
if $A_n=a$ then for any $m\geq 1$,
$$
|U^{\beta,u}_{T_{n+m}}(a)|\leq \left|\sum_{j=1}^mO_{n+j}\right| \leq \sum_{j=1}^m\left|O_{n+j}\right|=m.
$$
Therefore, if
$$
|\{a_0,A_1,...,A_{N-1}\}|=N,
$$
then
$$
|U^{\beta,u}_{T_{N-1}}(A_N)| \leq \max \left\{|U^{\beta,u}_{T_{N-1}}(a_0)|, |U^{\beta,u}_{T_{N-1}}(A_1)|, \ldots, |U^{\beta,u}_{T_{N-1}}(A_{N-1})|\right\} \leq
$$
$$
\max\{N-1,N-2, \ldots, 0\}=N-1.
$$
In other words, if on the first $N-1$ steps of the process we have $N-1$ different actors expression opinions, with all these actors being different from $a_0$,
then the absolute value of the social pressure on the actor expressing an opinion at instant $T_N$ is smaller than $N$. This implies that
$$
\inf\left\{n\geq 1: |U^{\beta,u}_{T_{n-1}}(A_{n})|< N\right\} \leq N.
$$

Now, if
$$
|\{a_0,A_1,...,A_{N-1}\}|\leq N-1,
$$
there exists $m\in \{1,...,N-1\}$ such that 
$$
A_m \in \{a_0,A_1,..., A_{m-1}\}.
$$
This implies that
$$
|U^{\beta,u}_{T_{m-1}}(A_{m})|\leq m \leq N-1,
$$
and therefore,
$$
\inf\left\{n\geq 1: |U^{\beta,u}_{T_{n-1}}(A_{n})|< N\right\} \leq m < N.
$$

\end{proof}

Define $T_0^<=T_0^>=0$
and for any $k\geq 1$,
$$
T_k^< \mydef \inf\{T_n>T_{k-1}^<: |U^{\beta,u}_{T_{n-1}}(A_{n})|<N\},
$$
$$
T_k^> \mydef \inf\{T_n>T_{k-1}^>: |U^{\beta,u}_{T_{n-1}}(A_{n})|\geq N\}.
$$
Lemma \ref{inftmenor} implies that $T_k^<$ is well defined for any $k\geq 1$. Now we can prove Part 1 of Theorem \ref{existence}.
 
\begin{proof}

To prove Part 1 of Theorem \ref{existence}, we will construct the process $(U_t^{\beta,u})_{t \in [0,+\infty)}$ with jump times $\{T_n: n\geq 1\}$ as the superposition of two process with jump times $\{T_k^<: k\geq 1\}$ and $\{T_k^>: k\geq 1\}$ in the following way. For any $v \in \st$, the jump rates of these two processes are
$$
q^<(v)\mydef 
\sum_{a \in A} \mathbf{1}\{|v(a)|< N\}(e^{\beta v(a)}+e^{-\beta v(a)})
$$
and
$$
q^>(v)\mydef 
\sum_{a \in A} \mathbf{1}\{|v(a)|\geq  N\}(e^{\beta v(a)}+e^{-\beta v(a)}).
$$
For any list $v \in \st$, we have that
$$
2\leq q^<(v) \leq \lambda,
$$
with
$$
\lambda \mydef N(e^{\beta(N-1)}+e^{-\beta(N-1)}).
$$

For any $v \in \st$, we define the functions $\Phi^{<,-}_v, \Phi^{<,+}_v, \Phi^{>,-}_v, \Phi^{>,+}_v: \{0\}\cup \A \to [0,\infty)$ inductively as follows. First, we have that
$$
\Phi^{<,-}_v(0)\mydef 0, \quad  \Phi^{<,+}_v(0)\mydef 0, \quad \Phi^{>,-}_v(0) \mydef \lambda \quad \text{ and } \quad \Phi^{>,+}_v(0)\mydef \lambda.
$$
For any $a \in \A$, define 
$$
\Phi^{<,-}_v(a)\mydef\Phi^{<,+}_v(a-1)+\mathbf{1}\{|v(a)| < N\} e^{-\beta v(a)},
$$
$$
\Phi^{<,+}_v(a)\mydef\Phi^{<,-}_v(a)+\mathbf{1}\{|v(a)| < N\} e^{+\beta v(a)},
$$
$$
\Phi^{>,-}_v(a)\mydef\Phi^{>,+}_v(a-1)+\mathbf{1}\{|v(a)| \geq N\} e^{-\beta v(a)},
$$
$$
\Phi^{>,+}_v(a)\mydef\Phi^{>,-}_v(a)+\mathbf{1}\{|v(a)| \geq  N\} e^{+\beta v(a)}.
$$

Note that for any $v\in \st$, we have that
\begin{multline*}
0=\Phi^{<,-}_v(0)=\Phi^{<,+}_v(0)\leq \Phi^{<,-}_v(1)\leq \Phi^{<,+}_v(1) \leq \ldots 
\\
\leq \Phi^{<,-}_v(N-1) \leq \Phi^{<,+}_v(N-1)\leq \Phi^{<,-}_v(N) \leq \Phi^{<,+}_v(N)=q^<(v) \leq \lambda.
\end{multline*}
Moreover,
\begin{multline*}
\lambda=\Phi^{>,-}_v(0)=\Phi^{>,+}_v(0)\leq \Phi^{>,-}_v(1)\leq \Phi^{>,+}_v(1) \leq \ldots 
\\
\leq \Phi^{>,-}_v(N-1) \leq \Phi^{>,+}_v(N-1)\leq \Phi^{>,-}_v(N) \leq \Phi^{>,+}_v(N)=\lambda+q^>(v).
\end{multline*}
Therefore, for any $v\in \st$, $\{\Phi^{<,-}_v(b),\Phi^{<,+}_v(b): b\in \{0\}\cup \A\}$ is a partition of the interval $[0,q^<(v))$ in which for any
$a\in\A$ such that $v(a) < N$, the length of $[\Phi^{<,+}_v(a-1), \Phi^{<,-}_v(a))$ is $e^{-\beta v(a)}$ (exactly the rate in which the actor $a$ expresses an opinion $-1$ associated to the list $v$) and the length of $[\Phi^{<,-}_v(a), \Phi^{<,+}_v(a))$ is $e^{\beta v(a)}$ (exactly the rate in which the actor $a$ expresses an opinion $+1$ associated to the list $v$). Similarly, $\{\Phi^{>,-}_v(b),\Phi^{>,+}_v(b): b\in \{0\}\cup \A\}$ is a partition of the interval $[\lambda, \lambda + q^>(v))$ in which for any
$a\in\A$ such that $v(a) \geq N$, the length of $[\Phi^{>,+}_v(a-1), \Phi^{>,-}_v(a))$ is $e^{-\beta v(a)}$ and the length of $[\Phi^{>,-}_v(a), \Phi^{>,+}_v(a))$ is $e^{\beta v(a)}$. 

Using these partitions indexed by $\st$ (see Figure \ref{fig:M2}), we now construct the process $(U_t^{\beta,u})_{t \in [0,+\infty)}$ as follows. Consider a homogeneous rate $1$ Poisson point process in the plane $[0,+\infty)^2$. Call $\mathcal{N}$ the counting measure of this process. Given the initial list $u \in \st$, define
$$
T_1 = \inf\left\{t>0: \mathcal{N}\Big((0,t]\times \left\{ [0,q^<(u)) \cup [\lambda,\lambda+q^>(u))\right\}\Big)=1\right\}.
$$
Denoting $R_1$ as the second coordinate of the mark $(T_1,R_1)$ of $\mathcal{N}$, we have
$$
\text{$A_1=b$, $O_1=-1$, \quad if } \quad R_1 \in [\Phi^{<,+}_{u}(b-1),\Phi^{<,-}_{u}(b)) \cup [\Phi^{>,+}_{u}(b-1),\Phi^{>,-}_{u}(b)),
$$
$$
\text{ $A_1=b$, $O_1=+1$, \quad if} \quad R_1 \in [
\Phi^{<,-}_{u}(b),
\Phi^{<,+}_{u}(b)) \cup [
\Phi^{>,-}_{u}(b),
\Phi^{>,+}_{u}(b)).
$$
At time $T_1$, we have $U_{T_1}^{\beta,u}=\pi^{A_1,O_1}(u)$. 

\begin{figure}[h]
\centering
\begin{tikzpicture}
\fill[yellow!80!white] (0,2.8) rectangle (9,3.5);

\fill[orange!80!white] (0,0) rectangle (9,2.8);
\fill[blue!40!white] (0,3.5) rectangle (9,8);

\draw[thick,dashed] (0,3.5)  node[anchor=east] {$\lambda$} -- (9.5,3.5);

\draw[thick,dashed] (0,8)  node[anchor=east] {$\Phi_u^{>,+}(5)$} -- (9.5,8);

\draw[thick,dashed] (0,3.75)   -- (9.5,3.75) node[anchor=west] {$\Phi_u^{>,-}(2)$};

\draw[thick,dashed] (0,4.7)  node[anchor=east] {$\Phi_u^{>,+}(2)$} -- (9.5,4.7);

\draw[thick,dashed] (0,6.2)   -- (9.5,6.2) node[anchor=north west] {$\Phi_u^{>,-}(4)$};

\draw[thick,dashed] (0,6.3)  node[anchor= south east] {$\Phi_u^{>,+}(4)$} -- (9.5,6.3);

\draw[thick,dashed] (0,7.8)   -- (9.5,7.8) node[anchor=west] {$\Phi_u^{>,-}(5)$};

\draw[thick,dashed] (0,0.4)  -- (9.5,0.4) node[anchor=west] {$\Phi_u^{<,-}(1)$};
\draw[thick,dashed] (0,0.8) node[anchor=east] {$\Phi_u^{<,+}(1)$} -- (9.5,0.8);
\draw[thick,dashed] (0,1.4)  -- (9.5,1.4) node[anchor=west] {$\Phi_u^{<,-}(3)$};

\draw[thick,dashed] (0,1.7) node[anchor=east] {$\Phi_u^{<,+}(3)$} -- (9.5,1.7);
\draw[thick,dashed] (0,2.5)  -- (9.5,2.5) node[anchor=west] {$\Phi_u^{<,-}(6)$};
\draw[thick,dashed] (0,2.8) node[anchor=east] {$\Phi_u^{<,+}(6)$} -- (9.5,2.8);

\draw[thick,->] (0,0) -- (9.5,0) node[anchor=north west] {};
\draw[thick,->] (0,0) -- (0,8.7) node[anchor=south east] {}; 

\draw[thick,dashed] (6,0)  node[anchor=north] {$T_1$} -- (6,8.3);

\draw [fill] (1.8,8.25) circle [radius=0.05];

\draw [fill] (3.8,3.25) circle [radius=0.05];

\draw [fill] (6,4.25) circle [radius=0.05];

\draw [fill] (7.3,2.25) circle [radius=0.05];

\draw [fill] (8.7,6.85) circle [radius=0.05];

\end{tikzpicture}
\caption{This figure represents the regions of  $[0,\infty) \times [0,\lambda+q^>(u)]$ and the marks of the Poisson point process $\mathcal{N}$ in $[0,\infty)^2$ considered on the construction of $T_1$. 
In this example, $N=6$ and the list $u \in \st$, satisfies $u(a)<N$ for $a=1,3,6$ and $u(a)\geq N$ for $a=2,4,5$. 
The height of each orange sub-region of the figure is exactly the rate in which one of the actors with social pressure smaller than $N$ (actors $1,3$ and $6$) expresses one of the possible opinions ($+1$ and $-1$). 
The height of each blue sub-region of the figure is exactly the rate in which one of the actors with social pressure greater or equal than $N$ (actors $2,4$ and $5$) expresses one of the possible opinions.
The points in the figure are the marks of $\mathcal{N}$ in this region.
$T_1$ is the time of the first mark of the Poisson point process inside the orange or blue area. Note that the marks of $\mathcal{N}$ higher than $\lambda+q^>(u)$ or in $[0,\infty]\times [q^<(u),\lambda)$ (the yellow region) are not considered to define $T_1$. In this example, the position of the mark at time $T_1$ (between $\Phi_u^{>,-}(2)$ and $\Phi_u^{>,+}(2)$) indicates that $A_1=2$ and $O_1=+1$.
} \label{fig:M2}
\end{figure}

More generally, for $n\geq 1$, we have
$$
T_n = \inf\left\{t>T_{n-1}: \mathcal{N}\Big((T_{n-1},t]\times \left\{ [0,q^<(U_{T_{n-1}}^{\beta,u})) \cup [\lambda,\lambda+q^>(U_{T_{n-1}}^{\beta,u}))\right\}\Big)=1\right\}.
$$
Denoting $R_n$ as the second coordinate of the mark $(T_n,R_n)$ of $\mathcal{N}$, we have 
$$
\text{$A_n=b$, $O_n=-1$, \quad if } \quad R_n \in \left[\Phi^{<,+}_{U_{T_{n-1}}^{\beta,u}}(b-1),\Phi^{<,-}_{U_{T_{n-1}}^{\beta,u}}(b)\right) \cup \left[\Phi^{>,+}_{U_{T_{n-1}}^{\beta,u}}(b-1),\Phi^{>,-}_{U_{T_{n-1}}^{\beta,u}}(b)\right),
$$ 
$$
\text{$A_n=b$, $O_n=+1$, \quad if} \quad R_n \in \left[
\Phi^{<,-}_{U_{T_{n-1}}^{\beta,u}}(b),
\Phi^{<,+}_{U_{T_{n-1}}^{\beta,u}}(b)\right) \cup \left[
\Phi^{>,-}_{U_{T_{n-1}}^{\beta,u}}(b),
\Phi^{>,+}_{U_{T_{n-1}}^{\beta,u}}(b)\right).
$$
At time $T_n$, we have $U_{T_n}^{\beta,u}=\pi^{A_n,O_n}(U_{T_{n-1}}^{\beta,u})$ 
(see Figure \ref{fig:M1}).

\begin{figure}[h]
\centering
\begin{tikzpicture}
\fill[yellow!80!white] (0,2) rectangle (6,3);
\fill[orange!80!white] (0,0) rectangle (6,2);
\fill[blue!40!white] (0,3) rectangle (6,4);

\draw[thick,dashed] (0,3)  node[anchor=east] {$\lambda$} -- (6.5,3);
\draw[thick,dashed] (0,4)  node[anchor=east] {$\lambda+q^>(u)$} -- (6.5,4);
\draw[thick,dashed] (0,2) node[anchor=east] {$q^<(u)$} -- (6.5,2);

\fill[yellow!80!white] (6,1.5) rectangle (9,3);
\fill[orange!80!white] (6,0) rectangle (9,1.5);
\fill[blue!40!white] (6,3) rectangle (9,5);

\draw[thick,dashed] (6,3)  node[anchor=east] {} -- (9,3);
\draw[thick,dashed] (6,5)  node[anchor=east] {$\lambda+q^>(U_{T_1}^{\beta,u})$} -- (9,5);
\draw[thick,dashed] (6,1.5) node[anchor=east] {$q^<(U_{T_1}^{\beta,u})$} -- (9,1.5);

\fill[yellow!80!white] (9,1.75) rectangle (10,3);
\fill[orange!80!white] (9,0) rectangle (10,1.75);
\fill[blue!40!white] (9,3) rectangle (10,4.2);

\draw[thick,dashed] (9,3)  node[anchor=east] {} -- (10,3);
\draw[thick,dashed] (9,4.2)  node[anchor=east] {$\lambda+q^>(U_{T_2}^{\beta,u})$} -- (10,4.2);
\draw[thick,dashed] (9,1.75) node[anchor=east] {$q^<(U_{T_2}^{\beta,u})$} -- (10,1.75);

\draw[thick,->] (0,0) -- (11,0) node[anchor=north west] {};
\draw[thick,->] (0,0) -- (0,6) node[anchor=south east] {};

\draw[thick,dashed] (6,0)  node[anchor=north] {$T_1$} -- (6,5.5);
\draw[thick,dashed] (9,0)  node[anchor=north] {$T_2$} -- (9,5.5);

\draw (10,-0.11) node[anchor=north] {$\ldots$};
\end{tikzpicture}
\caption{The regions of the plane 
$(T_n,T_{n+1})\times [0,q^{<}(U_{T_n}^{\beta,u})]$ and $(T_n,T_{n+1})\times[\lambda,\lambda+q^{>}(U_{T_n}^{\beta,u})]$, for $n=0,1,2,\ldots$.} \label{fig:M1}
\end{figure}

For any $n\geq 1$, we have that
$$
T_n^<= \inf\left\{T_m>T_{n-1}^<: R_m \in \big[0, q^<(U_{T_{m-1}}^{\beta,u})\big)\right\}, 
$$
$$
T_n^>= \inf\left\{T_m>T_{n-1}^>: R_m \in \big[\lambda, \lambda+q^>(U_{T_{m-1}}^{\beta,u})\big)\right\}. 
$$
Define $T_0^{\lambda} \mydef 0$ and for any $n\geq 1$,
$$
T_n^{\lambda} \mydef \inf\left\{t>T_{n-1}^{\lambda}: \mathcal{N}\Big((T_{n-1}^{\lambda},t]\times  [0,\lambda)\Big)=1\right\}.
$$
By construction, $\{T_n^{<}: n\geq 1\} \subset \{T_n^{\lambda}: n\geq 1\}$. Also, $\{T_n^{\lambda}: n\geq 1\}$ are the marks of a homogeneous Poisson point process with rate $\lambda$. Since $\P(\sup\{T_n^{\lambda}: n\geq 1\}=+\infty)=1$, we have that $\P(\sup\{T_n^{<}: n\geq 1\}=+\infty)=1$. By Lemma \ref{inftmenor}, the event $\sup\{T_n^{<}: n\geq 1\}=+\infty$ implies that $\sum_{n=1}^{+\infty} \mathbf{1}\{T_n^{>}\leq t\}<\infty$, for any $t>0$. 
$\{T_n: n\geq 1\}$ is the superposition of $\{T_n^<: n\geq 1\}$ and $\{T_n^>: n\geq 1\}$, and then, we conclude that $\P(\sup\{T_n: n\geq 1\}=+\infty)=1$.
\end{proof}

To prove Part 2 of Theorem \ref{existence}, we first show that starting from any initial list, the process has a positive and bounded probability to reach a fixed ladder list after $2N$ jumps of the process. With this, we are able to prove that $(U_t^{\beta,u})_{t \in [0,+\infty]}$ is an ergodic Markov process.

\begin{proof}
To prove Part $2$ of Theorem \ref{existence}, let $l \in \pladder$ satisfies $l(a)=a-1$, for all $a \in \A$.
For any $u \in \st$, we have that
$$
l=\pi^{1,+1} \circ \pi^{2,+1} \circ \ldots \circ \pi^{N,+1}(u),
$$
and then, if the event $\bigcap_{j=1}^{N}\{A_{j}=N-j+1, O_j=+1\}$ occurs, then
$
\tilde{U}^{\beta,u}_N=l.
$
This implies that for any $u \in \st$,
\begin{equation} \label{eq:l}
\P\left(\tilde{U}_{N}^{\beta,u}=l\right) >0.
\end{equation}

For any $u' \in \mathcal{S}$,
\begin{align*}
    &\P\Big(\tilde{U}^{\beta,u}_{n+2N}=l\ \Big|\ \tilde{U}^{\beta,u}_n=u'\Big) \geq \\ 
    &\P\left(\bigcap_{a \in \mathcal{A}}\{|\tilde{U}_{n+N}^{\beta,u}(a)|< N\} \Big| \ \tilde{U}_n^{\beta,u}=u'\right)\P\left(\tilde{U}_{n+2N}^{\beta,u}=l \ \Big|\ \bigcap_{a \in \mathcal{A}}\{|\tilde{U}_{n+N}^{\beta,u}(a)|< N\}\right).
\end{align*}
For any $u \in \st$, note that
$$
\P\left(\bigcap_{a \in \A}\{ |\tilde{U}_{N}^{\beta,u}(a)|< N  \} \  \Big| \ \bigcap_{j=1}^NM_j^u\right) =1.
$$
This together with Proposition \ref{teo:m1m} implies that
$$
 \P\left(\bigcap_{a \in \A}\{ |\tilde{U}_{N}^{\beta,u}(a)|< N  \}\right) \geq \P\left(\bigcap_{j=1}^NM_j^u\right) \geq \zeta_{\beta}^N.
$$
Also, 
there exists $\epsilon^*>0$ such that
$$
\P\left(\tilde{U}_{n+2N}^{\beta,u}=l \ \Big|\ \bigcap_{a \in \mathcal{A}}\{|\tilde{U}_{n+N}^{\beta,u}(a)|< N\}\right) \geq 
$$
$$
\min\left\{\P\left(\tilde{U}_{n+2N}^{\beta,u}=l \ \Big|\ \tilde{U}_{n+N}^{\beta,u}=v\right) : v \in \st, \bigcap_{a \in \mathcal{A}}\{|v(a)|< N\} \right\} = \epsilon^*.
$$
Note that $\epsilon^*>0$ since 
it is the minimum of a finite set of positive numbers, by \eqref{eq:l}.
Therefore, for any $u' \in \st$,
$$
\P\Big(\tilde{U}^{\beta,u}_{n+2N}=l\ \Big|\ \tilde{U}^{\beta,u}_n=u'\Big) \geq \zeta_{\beta}^N\epsilon^*.
$$

Recall that $\tilde{R}^{\beta,l}(l)=\inf\{n\geq 1:\tilde{U}^{\beta,l}_n=l\}.$
The last inequality implies that for any $t>0$,
$$
\P(\tilde{R}^{\beta,l}(l)>t) \leq \P(2N \times \text{Geom}(\zeta_{\beta}^N\epsilon^*)>t),
$$
where $\text{Geom}(r)$ denotes a random variable with geometric distribution assuming values in $\{1,2,\ldots\}$ and with parameter $r \in (0,1)$. This implies that $\mathbb{E}(\tilde{R}^{\beta,l}(l))< +\infty$ and then,
$(\tilde{U}^{\beta,u}_n)_{n\geq 0}$ is a positive-recurrent Markov chain.

The jump rate of the process $(U_t^{\beta,u})_{t\in [0,+\infty)}$ satisfies
$$
\sum_{a \in \A}(e^{\beta v(a)}+e^{-\beta v(a)}) \geq 2N,
$$
for any $v \in \st$. Putting all this together we conclude that $(U_t^{\beta,u})_{t\in [0,+\infty)}$ is ergodic.

\end{proof}



\section{Proof of Theorem \ref{fastconsensus}} \label{sec: fastconsensus}


To prove Part $1$ of Theorem \ref{fastconsensus}, we first study the invariant probability measure of the embedded process $(\tilde{U}_n^{\beta,u})_{n\geq 0}$. We prove that the invariant measure of a list $u \notin \ladder$ such that $\max\{|u(a)|: a \in \A\} < N$ is small, decreasing when $\beta$ increases. We use the fact that, starting from $u$, the embedded process quickly reaches the set of ladder lists, taking a long time to return to the initial list $u$. This is the content of Proposition \ref{decay}. With this we are able to prove Part $1$ of Theorem \ref{fastconsensus}, using the relation between the invariant probability measure of $(\tilde{U}_n^{\beta,u})_{n\geq 0}$ and $(U_t^{\beta,u})_{t\in [0,+\infty)}$. 


\begin{proposition} \label{decay}
For any $\beta \geq 0$ and $u \notin \ladder$ such that $\max\{|u(a)|: a \in \A\} < N$ we have that
    $$
    \tilde{\mu}^{\beta}(u) \leq C'e^{-\beta(N-1)},
    $$
    where $C'=C'(N)>0$.
\end{proposition}

\begin{proof}
First, we have that
\begin{equation} \label{eq:decay1}
\mathbb{E}(\tilde{R}^{\beta,u}(u)) =\sum_{m=1}^{+\infty}\P(\tilde{R}^{\beta,u}(u)\geq m) \geq \sum_{m=3(N-1)}^{+\infty}\P(\tilde{R}^{\beta,u}(u)\geq m).
\end{equation}

For any $u \notin \ladder$ such that $\max\{|u(a)|: a \in \A\} < N$, by Proposition \ref{teo:m1m} and Corollary \ref{coro:return},
\begin{equation} \label{eq:prop11_1}
 \P\left(\{\tilde{U}_{n}^{\beta,u} \neq u, \text{ for } n=1,2,...,3(N-1)\}\cap \{\tilde{U}_{3(N-1)}^{\beta,u} \in\ladder\}\right) \geq \P\left(\bigcap_{j=1}^{3(N-1)} M_j^u\right).
\end{equation}

Note that if 
$$
v\in \{u \in \st: 0\leq u(a_1) < u(a_2) < \ldots < u(a_N), \text{ for } \{a_1,\ldots, a_N\} = \A\}
$$
and $O_1=+1$, then
\begin{equation} \label{eq:prop11eq}
\tilde{U}_1^{\beta,v}\in \{u \in \st: 0\leq u(a_1) < u(a_2) < \ldots < u(a_N), \text{ for } \{a_1,\ldots, a_N\} = \A\}.
\end{equation}
Moreover, note that 
$$
\pladder \subset \{u \in \st: 0\leq u(a_1) < u(a_2) < \ldots < u(a_N), \text{ for } \{a_1,\ldots, a_N\} = \A\}
$$
and note that if 
$$
v \in \{u \in \st: 0\leq u(a_1) < u(a_2) < \ldots < u(a_N), \text{ for } \{a_1,\ldots, a_N\} = \A\} \setminus \pladder,
$$
then $\max\{|v(a)|: a \in \A\}\geq N$. By the symmetric properties of the process, this implies that for any $l \in \ladder$ and for any $m\geq 1$, if
$$
\bigcap_{j=1}^m 
\{O_j \tilde{U}_{j-1}^{\beta,l}(A_{j})>0\}
$$
occurs, then 
$$
\tilde{U}_{j}^{\beta,l} \in \ladder \cup \{u \in \st: \max\{|u(a)|: a \in \A\}\geq N\}
$$
for all $j=1,\ldots, m$. 

In other words, starting from a ladder list, if the process has a sequence of opinions expressed by actors with non-null social pressure and in the same direction of their social pressure, then either the process is in a ladder list or there is at least one actor with social pressure greater than or equal to $N$.
Since $u \notin \ladder$ and  $\max\{|u(a)|: a \in \A\} < N$, for any $l \in \ladder$ and for any $m\geq 1$, we have that
\begin{equation} \label{eq:prop11_2}
\P\left(\{\tilde{U}_{n}^{\beta,l} \neq u, \text{ for } n=1,2,...,m\}\right)\geq \P\left(\bigcap_{j=1}^m 
\{O_j \tilde{U}_{j-1}^{\beta,l}(A_{j})>0\}\right).
\end{equation}
Therefore, for any $m\geq 1$, putting together \eqref{eq:prop11_1}, \eqref{eq:prop11_2} and the Markov property, we have that
\begin{equation} \label{eq:decay2}
\P(\tilde{R}^{\beta,u}(u)\geq 3(N-1)+m)\geq \P\left(\bigcap_{j=1}^{3(N-1)}M_j^u\right)
\P\left(\  \bigcap_{j=1}^{m} \{O_j \tilde{U}_{j-1}^{\beta,l}(A_{j})>0\}\right).
\end{equation}

Note that for any
$$
v\in \{u \in \st: 0\leq u(a_1) < u(a_2) < \ldots < u(a_N), \text{ for } \{a_1,\ldots, a_N\} = \A\},
$$
we have that
$$
\P\left(
\{O_1 \tilde{U}_{0}^{\beta,v}(A_{1})>0\}\right) \geq \eta,
$$
where
$$
\eta \mydef \frac{\displaystyle\sum_{j=1}^{N-1}e^{\beta j}}{\displaystyle\sum_{j=0}^{N-1}(e^{\beta j}+e^{-\beta j})}.
$$
Therefore, for any $l \in \ladder$, putting together \eqref{eq:prop11eq}, the Markov property and the symmetric properties of the process, we have that
\begin{equation} \label{eq:decay3}
\P\left(\bigcap_{j=1}^m 
\{O_j \tilde{U}_{j-1}^{\beta,l}(A_{j})>0\}\right) \geq \eta^m.
\end{equation}
By Proposition \ref{teo:m1m} and Equations \eqref{eq:decay1}, \eqref{eq:decay2} and  \eqref{eq:decay3}, it follows that
$$
\mathbb{E}(\tilde{R}^{\beta,u}(u))\geq 
\P\left(\bigcap_{j=1}^{3(N-1)}M_j^u\right) \left( 1+\sum_{m=1}^{\infty}\eta^m \right) \geq \frac{(2N)^{-3(N-1)}}{1-\eta}.
$$
From the  classical Kac's Lemma (see \cite{kac}), 
we conclude that
$$
\tilde{\mu}^{\beta}(u) \leq C'e^{-\beta(N-1)},
$$
where  $C'=(N+1)(2N)^{3(N-1)}$.
\end{proof}

\begin{corollary} \label{zerodecay}
$\tilde{\mu}^{\beta}(\vec{0})<C'_1e^{-N\beta}$ for all 
$\beta\geq 0$, where $C'_1=C'_1(N)>0$. 
\end{corollary}

\begin{proof}
Let $v \in \st$ be the list in which $v(1)=0$ and $v(a)=1, \text{ for } a \in \A \setminus\{1\}$. We have that
$$
\{u \in \st:
\P(\tilde{U}^{\beta,u}_1=\vec{0})>0\}=\{\sigma(v),-\sigma(v) \in \st \text{ such that } \sigma: \A \to \A \text{ is a bijective map} \}.
$$
Therefore, by the symmetric properties of the process, we have that
$$
\tilde{\mu}^{\beta}(\vec{0})=\sum_{u \in \st}\P(\tilde{U}^{\beta,u}_1=\vec{0})\tilde{\mu}^{\beta}(u)=2N\P(\tilde{U}^{\beta,v}_1=\vec{0})\tilde{\mu}^{\beta}(v).
$$
Since
$$
\P(\tilde{U}^{\beta,v}_1=\vec{0})=\frac{1}{2+(N-1)(e^{+\beta}+e^{-\beta})} \leq e^{-\beta},
$$
from Proposition \ref{decay} we conclude that
$$
\tilde{\mu}^{\beta}(\vec{0})\leq (N+1)(2N)^{3(N-1)}e^{-\beta(N-1)}(2N)e^{-\beta}=C'_1e^{-\beta N}.
$$
\end{proof}

Now we can prove the Part 1 of Theorem \ref{fastconsensus}. 

\begin{proof}
To prove  Part 1 of Theorem \ref{fastconsensus}, first note that for any $u \in \st$, the invariant measure $\mu^{\beta}$ satisfies
$$
\mu^{\beta}(u)=\frac{\tilde{\mu}^{\beta}(u)}{q_{\beta}(u)}\left(\sum_{u' \in \st}\frac{\tilde{\mu}^{\beta}(u')}{q_{\beta}(u')}\right)^{-1},
$$
where for any $v \in \st$,
$$
q_{\beta}(v) \mydef \sum_{a \in \A}(e^{\beta v(a)}+e^{-\beta v(a)})
$$
is the jump rate of $(U_t^{\beta,u})_{t\in[0,+\infty)}$ at list $v$. 
Note that for any $l\in \ladder,$ 
$$
q_{\beta}(l)=\sum_{j=0}^{N-1}(e^{\beta j}+e^{-\beta j}).
$$
Therefore,
$$
\mu^{\beta}(\ladder)=\sum_{u \in \ladder}\frac{\tilde{\mu}^{\beta}(u)}{q_{\beta}(u)}\left(\sum_{u' \in \st}\frac{\tilde{\mu}^{\beta}(u')}{q_{\beta}(u')}\right)^{-1}=
\frac{1}{1+ \displaystyle\frac{
\sum_{j=0}^{N-1}(e^{\beta j}+e^{-\beta j})}{\tilde{\mu}^{\beta}(\ladder)}\sum_{u' \not\in \ladder}\frac{\tilde{\mu}^{\beta}(u')}{q_{\beta}(u')}}.
$$
By Proposition \ref{teo:m1m}, we have that
$$
\tilde{\mu}^{\beta}(\ladder)=\sum_{u \in \st} \tilde{\mu}^{\beta}(u)\P(\tilde{U}_{3(N-1)}^{\beta,u} \in \ladder) \geq (2N)^{-3(N-1)}.
$$
We also have that
$$
\sum_{j=0}^{N-1}(e^{\beta j}+e^{-\beta j}) \leq 2Ne^{\beta(N-1)},
$$
and then,
$$
\mu^{\beta}(\ladder) \geq  \frac{1}{1+ \displaystyle
(2N)^{3(N-1)+1}e^{\beta(N-1)}\sum_{u' \not\in \ladder}\frac{\tilde{\mu}^{\beta}(u')}{q_{\beta}(u')}}.
$$

For any $u \in \st$ such that $\max\{|u(a)|: a \in \A\} \geq N$, we have that $q_{\beta}(u) \geq e^{\beta N}$, and then
$$
\frac{\tilde{\mu}^{\beta}(u)}{q_{\beta}(u)} \leq \tilde{\mu}^{\beta}(u)e^{-\beta N}.
$$
Observe that for any $u \neq \vec{0}$, $q_{\beta}(u) \geq e^{\beta}$. Therefore,
by Proposition \ref{decay} and Corollary \ref{zerodecay}, it follows that for any $u \not\in \ladder$ such that $\max\{|u(a)|: a \in \A\} < N$, we have that
$$
\frac{\tilde{\mu}^{\beta}(u)}{q_{\beta}(u)} \leq C'_1e^{-\beta N}.
$$
Moreover, $|\{u \not\in \ladder : \max\{|u(a)|: a \in \A\} < N\}|\leq N(2N-1)^{N-1}$.
Putting all this together, we have that
$$
\sum_{u' \not\in \ladder}\frac{\tilde{\mu}^{\beta}(u')}{q_{\beta}(u')} \leq e^{-\beta N} +  N(2N-1)^{N-1}C'_1e^{-\beta N}.
$$
We conclude that
$$
\mu^{\beta}(\ladder) \geq \frac{1}{1+Ce^{-\beta}} 
\geq 1-Ce^{-\beta},
$$
where
$$
C=(2N)^{3(N-1)+1}[1+C'_1N(2N-1)^{N-1}].
$$

\end{proof}

To prove Part 2 of Theorem \ref{fastconsensus} we use the fact that the embedded process quickly reaches the set of ladder lists and the exit rate of any non-null list is bigger than $e^{\beta}$.

\begin{proof} To prove Part 2 of Theorem \ref{fastconsensus}, we first note that for any $u \in \st \setminus \{\vec{0}\}$, we have that
\begin{equation} \label{eqteo2p1}
\P\Big(R^{\beta,u}(\ladder)>t\Big)\leq
\P\left(R^{\beta,u}(\ladder)>t\ , \ \bigcap_{j=1}^{3(N-1)}M_j^u\right)+\P\left(\bigcup_{j=1}^{3(N-1)}(M_j^u)^c\right).
\end{equation}
Part $2$ of Proposition \ref{teo:m1m} implies that the right-hand side of Equation  \eqref{eqteo2p1} is bounded above by
$$
\P\left(T_{3(N-1)}>t\ , \ \bigcap_{j=1}^{3(N-1)}M_j^u\right)+\P\left(\bigcup_{j=1}^{3(N-1)}(M_j^u)^c\right).
$$
By Part $1$ of Proposition \ref{teo:m1m}, we have that
\begin{equation}\label{eqteo21}
\P\left(\bigcup_{j=1}^{3(N-1)}(M_j^u)^c\right) \leq 1-\left(\frac{e^{\beta}}{e^{\beta}+e^{-\beta}+2(N-1)}\right)^{3(N-1)}.
\end{equation}
Note that the exit rate of any list different from the null list is always bigger than $e^{\beta}$.  Note also that for any $u \neq \vec{0}$, the event $M_1^{u}$ implies that $U_{T_1}^{\beta,u} \neq \vec{0}$.
Therefore,
\begin{equation}\label{eqteo22}
\P\left(T_{3(N-1)}>t\ , \ \bigcap_{j=1}^{3(N-1)}M_j^u\right)\leq \P\left(\sum_{n=1}^{3(N-1)}E_n> t\right),
\end{equation}
where $(E_n)_{n\geq 1}$ is an i.i.d. sequence of exponentially distributed random variables with mean $1/e^{\beta}$. It follows that
$$
\P\left(\sum_{n=1}^{3(N-1)}E_n> t\right) \leq \P\left(\bigcup_{n=1}^{3(N-1)}\left\{E_n> \frac{t}{3(N-1)}\right\}\right) \leq 3(N-1)e^{-e^{\beta}t/(3(N-1))}.
$$
We conclude the proof by taking $t=e^{-\beta(1-\delta)}$ 
and noting that the bounds in \eqref{eqteo21} and \eqref{eqteo22} do not depend on $u$.
\end{proof}

The next corollary presents an
equivalent result of Theorem \ref{fastconsensus} when the initial list is the null list.

\begin{corollary} \label{coro: fastconsensus}
For any fixed  $\delta>0$,
$$
\P\Big(R^{\beta,\vec{0}}(\ladder)>\tau+e^{-\beta(1-\delta)}\Big) \to 0 \text{ as } \beta \to \infty,
$$
where $\tau$ exponentially distributed random time with mean $1/2N$ independent from $(U_t^{\beta,\vec{0}})_{t \in [0,+\infty)}$.
\end{corollary}
\begin{proof}
Corollary \ref{coro: fastconsensus} follows directly from Part $2$ of Theorem \ref{fastconsensus} and the fact that given the initial list $\vec{0}$, the first jump time $T_1$ is an exponentially distributed random time with mean $1/2N$.
\end{proof}

\begin{remark} \label{remarkteo2}
Following the same steps of the proof of Part 2 of Theorem \ref{fastconsensus}, it follows that for any fixed $\delta>0$ 
$$
\inf_{u \in \pcons}\P\Big(R^{\beta,u}(\pladder) < \min\{e^{-\beta(1-\delta)}, R^{\beta,u}(\ncons)\}\Big) \to 1 \text{, as } \beta \to +\infty.
$$
\end{remark}

\section{Proof of Theorem \ref{metastable}}\label{cap:metastability}

Recall that in Section \ref{sec:auxiliary}, for any $l \in \pladder$ and for any $\beta \geq 0$ we considered $c_{\beta}$ as the positive real number such that
\begin{equation*}
\P(
R^{\beta,l}(\nladder)>c_{\beta})=e^{-1}.
\end{equation*}

To prove Theorem \ref{metastable}, we prove in Proposition \ref{explim2} that for $l\in \pladder$,
$$
\frac{R^{\beta,l}(\nladder)}{c_{\beta}}\to \text{Exp}(1), \text{ as } \beta \to +\infty,
$$
where $\text{Exp}(1)$ is a random variable exponentially distributed with mean $1$. We will prove that the limiting distribution satisfies the memoryless property, which characterizes the exponential distribution. For this, we use the fact that $c_{\beta} \to +\infty$ as $\beta \to +\infty$, which is the content of Proposition \ref{cbetabound}. Lemmas \ref{lemmasup} and \ref{prop7} give the necessary conditions to replace $c_{\beta}$ by $\mathbb{E}[R^{\beta,l}(\nladder)]$ in Proposition \ref{explim2}. Using the fact that the process starting in $\pcons$ will quickly reaches $\pladder$, as $\beta \to +\infty$, we finish the proof of Theorem \ref{metastable}.

\begin{proposition}\label{explim2} For any $l \in \pladder$
$$
\frac{R^{\beta,l}(\nladder)}{c_{\beta}}\to \text{Exp}(1), \text{ in distribution as } \beta \to +\infty,
$$
where $\text{Exp}(1)$ is a random variable exponentially distributed with mean $1$.
\end{proposition}

To prove Proposition \ref{explim2}, we will first prove the following lemma.

\begin{lemma} \label{propcoupling}
For any $\beta\geq 0$, for any $l \in \ladder$ and for any $s>0$,
$$
\P(U_s^{\beta,l} \in \st \setminus\ladder) \leq\frac{1-\mu^{\beta}(\ladder)}{\mu^{\beta}(\ladder)}.
$$
\end{lemma}
\begin{proof} For any $s>0$, by definition,
$$
\mu^{\beta}(\ladder)=\sum_{u \in \ladder}\mu^{\beta}(u)\P(U^{\beta,u}_s \in \ladder)+\sum_{u \in \st \setminus \ladder}\mu^{\beta}(u)\P(U^{\beta,u}_s \in \ladder).
$$
By the symmetric properties of the process, it follows that for any $l,l' \in \ladder$,
$$
\P(U^{\beta,l}_s \in \ladder)=\P(U^{\beta,l'}_s \in \ladder).
$$
Moreover,
$$
\sum_{u \in \st \setminus \ladder}\mu^{\beta}(u)\P(U^{\beta,u}_s \in \ladder) \leq 1-\mu^{\beta}(\ladder).
$$
Putting all this together we have that
$$
\mu^{\beta}(\ladder) \leq \mu^{\beta}(\ladder)\P(U^{\beta,l}_s \in \ladder)+(1-\mu^{\beta}(\ladder)).
$$
We conclude the proof of Lemma \ref{propcoupling} by rearranging the terms of this last inequality.

\end{proof}

\begin{proof} We will now prove Proposition \ref{explim2}.
First of all, we will prove that for any $l \in \pladder$ and for any pair of positive real numbers $s,t\geq 0$, the following holds
\begin{equation}  \label{explim1}
 \lim_{\beta \to +\infty}\left|\P\left(\frac{R^{\beta,l}(\nladder)}{c_{\beta}}>s+t \right)-\P\left(\frac{R^{\beta,l}(\nladder)}{c_{\beta}}>s \right)\P\left(\frac{R^{\beta,l}(\nladder)}{c_{\beta}}>t \right) \right|=0.   
\end{equation}

To simplify the presentation of the proof,  we will use the shorthand notation $R^{\beta,u}$ instead of $R^{\beta,u}(\nladder)$, for any $u \notin \nladder$.

Indeed, for a fixed $l \in \pladder$, by the Markov property and the triangle inequality, 
$$
\left|\P\left(\frac{R^{\beta,l}}{c_{\beta}}>s+t \right)-\P\left(\frac{R^{\beta,l}}{c_{\beta}}>s \right)\P\left(\frac{R^{\beta,l}}{c_{\beta}}>t \right) \right|=
$$
$$
\left|\sum_{u \in \st\setminus\nladder}\P\left(U_{c_{\beta}s }^{\beta,l}=u,\frac{R^{\beta,l}}{c_{\beta}}>s+t \right)-\P\left(U_{c_{\beta}s }^{\beta,l}=u,\frac{R^{\beta,l}}{c_{\beta}}>s \right)\P\left(\frac{R^{\beta,l}}{c_{\beta}}>t \right) \right|\leq
$$
\begin{equation} \label{eqmeta1}
\sum_{u \in \st\setminus\nladder}\P\left(U_{c_{\beta}s }^{\beta,l}=u,\frac{R^{\beta,l}}{c_{\beta}}>s\right)\left|\P\left(\frac{R^{\beta,u}}{c_{\beta}}>t \right)-\P\left(\frac{R^{\beta,l}}{c_{\beta}}>t \right)\right|. 
\end{equation}

By the symmetric properties of the process, for any $u \in \pladder$,
$$
\P\left(\frac{R^{\beta,u}}{c_{\beta}}>t \right)=\P\left(\frac{R^{\beta,l}}{c_{\beta}}>t \right).
$$
Therefore, 
the right-hand  side of Equation \eqref{eqmeta1} is bounded above by
\begin{equation} \label{eqpropexp2}
\sum_{u \in \st\setminus\ladder}\P\left(U_{c_{\beta}s }^{\beta,l}=u,\frac{R^{\beta,l}}{c_{\beta}}>s\right)\left|\P\left(\frac{R^{\beta,u}}{c_{\beta}}>t \right)-\P\left(\frac{R^{\beta,l}}{c_{\beta}}>t \right)\right|\leq \P\left(U_{c_{\beta}s }^{\beta,l}\in \st\setminus\ladder\right). 
\end{equation}
By Lemma \ref{propcoupling}, Equation \eqref{eqpropexp2} and Theorem \ref{fastconsensus} implies \eqref{explim1}.


By definition,
$$
\P\left(\frac{R^{\beta,l}(\nladder)}{c_{\beta}} > 1\right)=e^{-1}.
$$
Iterating \eqref{explim1} with $t=s=2^{-n}$, for $n=1,2,\ldots$, we have that
\begin{equation*} 
\P\left(\frac{R^{\beta,l}(\nladder)}{c_{\beta}} > 2^{-n}\right) \to e^{-2^{-n}}, \text{ as } \beta \to +\infty.
\end{equation*}
More generally,  this iteration implies that for any
$$
t \in \left\{\sum_{n=1}^{m}b(n)2^{-n}: b(n) \in \{0,1\}, n=1,...,m, m \geq 1\right\}
$$
is valid that
\begin{equation}
\P\left(\frac{R^{\beta,l}(\nladder)}{c_{\beta}} > t\right) \to e^{-t}, \text{ as } \beta \to +\infty.
\label{convergence}
\end{equation}
Any real number $r \in (0,1)$ has a binary representation
$$
r= \sum_{n=1}^{+\infty}b(n)2^{-n},
$$
where for any $n\geq 1$, $b(n) \in \{0,1\}$. Therefore, 
the monotonicity of 
$$
t \to \P\left(\frac{R^{\beta,l}(\nladder)}{c_{\beta}} > t\right)
$$
implies that the convergence in \eqref{convergence} is valid for any $t \in (0,1)$. Moreover, for any positive integer $n\geq 1$,
Equation \eqref{explim1} implies that
\begin{equation*}
    \P\left(\frac{R^{\beta,l}(\nladder)}{c_{\beta}} > n\right) \to e^{-n}, \text{ as } \beta \to +\infty.
\end{equation*}
Putting all this together, we conclude that \eqref{convergence} is valid for any $t >0$.
\end{proof}

\begin{remark} \label{remark1}For any $l \in \pladder$ and for any $\beta \geq 0$, the function $f_{\beta}:[0,+\infty) \to [0,1]$ given by
$$
f_{\beta}(t)=\P\left(\frac{R^{\beta,l}(\nladder)}{c_{\beta}} > t \right)
$$is monotonic. Also, by Proposition \ref{explim2}, it converges pointwise as $\beta \to +\infty$ to a continuous function. Therefore, given $\epsilon_{\beta}>0$ such that $\displaystyle\lim_{\beta \to +\infty}\epsilon_{\beta}=0$, for any $t>0$ we have that
$$
\lim_{\beta \to +\infty}\P\left(\frac{R^{\beta,l}(\nladder)}{c_{\beta}} >t+\epsilon_{\beta} \right)=\lim_{\beta \to +\infty}\P\left(\frac{R^{\beta,l}(\nladder)}{c_{\beta}} > t-\epsilon_{\beta}\right)=e^{-t}.
$$

\end{remark}

To prove Theorem \ref{metastable}, we will first prove the two following lemmas.

\begin{lemma} \label{lemmasup}
For any $u \in \st$, let
$$
\rho_u \mydef \P\left(R^{\beta,u}(\pladder)<R^{\beta,u}(\nladder) \Big| \bigcap_{j=1}^{3(N-1)}M_j^u\right).
$$
Then, for any $t>0$,
$$
\lim_{\beta \to +\infty}\sup_{u \in \st \setminus\nladder}\left|\P\left(\frac{R^{\beta,u}(\nladder)}{c_{\beta}} > t\right)-e^{-t}\rho_u\right|=0.
$$
\end{lemma}
\begin{proof}
Denoting $E_{\beta,u}=\{R^{\beta,u}(\ladder)<e^{\beta/2},\bigcap_{j=1}^{3(N-1)}M_j^u\}$, we have that for any $u \not\in \ladder$,
\begin{equation}\label{eqsup}
\P\left(\frac{R^{\beta,u}(\nladder)}{c_{\beta}} > t\right) = \P\left(\frac{R^{\beta,u}(\nladder)}{c_{\beta}} > t,E_{\beta,u}, R^{\beta,u}(\pladder)<R^{\beta,u}(\nladder)\right)+
\end{equation}
$$
\P\left(\frac{R^{\beta,u}(\nladder)}{c_{\beta}} > t,E_{\beta,u}, R^{\beta,u}(\nladder)<R^{\beta,u}(\pladder)\right)+\P\left(\frac{R^{\beta,u}(\nladder)}{c_{\beta}} > t,E_{\beta,u}^c\right).
$$
By Proposition \ref{cbetabound}, there exists $\beta_t>0$ such that for any $\beta > \beta_t$, $c_{\beta}t> e^{\beta/2}$. By the definition of $E_{\beta,u}$, this implies that, for any $\beta > \beta_t$,
$$
\P\left(\frac{R^{\beta,u}(\nladder)}{c_{\beta}} > t,E_{\beta,u}, R^{\beta,u}(\nladder)< R^{\beta,u}(\pladder)\right)\leq \P(c_{\beta}t <R^{\beta,u}(\nladder)<e^{\beta/2}) = 0.
$$

Note that the first term on the right-hand side of Equation \eqref{eqsup} is equal
$$
\P\left(\frac{R^{\beta,u}(\nladder)}{c_{\beta}} > t\ | \ E_{\beta,u}, R^{\beta,u}(\pladder)<R^{\beta,u}(\nladder)\right) \times \rho_u \times \P(E_{\beta,u}).
$$
Considering $l \in \pladder$, by the Markov property it follows that the term above is bounded below by
\begin{equation}\label{eqlb}
\P\left(\frac{R^{\beta,l}(\nladder)}{c_{\beta}} > t \right)\rho_u[1-\P(E_{\beta,u}^c)] 
\end{equation}
and bounded above by
$$
\P\left(\frac{R^{\beta,l}(\nladder)}{c_{\beta}} > t-\frac{e^{\beta/2}}{c_{\beta}} \right)\rho_u.
$$

Therefore, considering $l \in \pladder$, for any $u \in \st \setminus\nladder$ and for all $\beta>\beta_t$, the left-hand side of Equation \eqref{eqsup} is bounded below by \eqref{eqlb}
and bounded above by
\begin{equation}\label{equb}
\P\left(\frac{R^{\beta,l}(\nladder)}{c_{\beta}} > t-\frac{e^{\beta/2}}{c_{\beta}} \right)\rho_u+\P(E_{\beta,u}^c). 
\end{equation}
By Theorem \ref{fastconsensus} and Corollary \ref{coro: fastconsensus},
$$
\lim_{\beta \to +\infty}\sup_{u \in \st \setminus \nladder}\P\left(E_{\beta,u}^c\right)=0.
$$
By Proposition $\ref{cbetabound}$, it follows that $\displaystyle\lim_{\beta \to +\infty}e^{\beta/2}/c_{\beta}=0$. Therefore, by Remark \ref{remark1} we have that
$$
\lim_{\beta \to +\infty}\P\left(\frac{R^{\beta,l}(\nladder)}{c_{\beta}} > t \right)=
\lim_{\beta \to +\infty}\P\left(\frac{R^{\beta,l}(\nladder)}{c_{\beta}} > t-\frac{e^{\beta/2}}{c_{\beta}} \right)
=e^{-t}.
$$
We conclude the proof by noting that the limits in the last equation do not depend on $u$.

\end{proof}

\begin{lemma} \label{prop7}
There exists $\gamma\in (0,1)$ and $\beta_{\gamma}>0$ such that for any $\beta > \beta_{\gamma}$ and any $l \in \ladder^+$, the following upperbound holds
$$
\P\left(\frac{R^{\beta,l}(\nladder)}{c_{\beta}} > n\right) \leq \gamma^n,
$$
 for any positive integer $n\geq 1$.
\end{lemma}

\begin{proof}
Taking $t=1$ on Lemma \ref{lemmasup}, we have that for any fixed $\gamma \in (e^{-1},1)$,  there exists $\beta_{\gamma}$ such that for all $\beta>\beta_{\gamma}$ and for any $u \notin \nladder$, 
\begin{equation} \label{p7eq1}
\P\left(\frac{R^{\beta,u}(\nladder)}{c_{\beta}}>1 \right)\leq \gamma <1.
\end{equation}
 For any $l \in \pladder$ and for any $n \in \{2,3,\ldots\}$, by Markov property,
$$
\P\left(\frac{R^{\beta,l}(\nladder)}{c_{\beta}}>n \right)=\sum_{u \in \st \setminus \nladder} \P\left(\frac{R^{\beta,l}(\nladder)}{c_{\beta}}>n-1,U_{c_{\beta}(n-1)}^{\beta,l}=u \right)\P\left(\frac{R^{\beta,u}(\nladder)}{c_{\beta}}>1\right).
$$
Equation \eqref{p7eq1} implies that for any $\beta >\beta_{\gamma}$,
\begin{equation}\label{indu}
\P\left(\frac{R^{\beta,l}(\nladder)}{c_{\beta}}>n\right)\leq 
\gamma\P\left(\frac{R^{\beta,l}(\nladder)}{c_{\beta}}>n-1\right).
\end{equation}
We finish the proof by iterating \eqref{indu}.
\end{proof}

\begin{proof} We will now prove Theorem \ref{metastable}.

First of all, we will prove that for any $l \in \pladder$, the following holds
\begin{equation} \label{p1t3}
\frac{R^{\beta,l}(\nladder)}{\mathbb{E}[R^{\beta,l}(\nladder)]}\to \text{Exp}(1) \text{ in distribution, as } \beta \to +\infty.
\end{equation}
Considering Proposition  \ref{explim2}, we only need to show that
$$
\lim_{\beta \to +\infty}\frac{\mathbb{E}[R^{\beta,l}(\nladder)]}{c_{\beta}}=1.
$$
Actually, 
$$
\lim_{\beta \to +\infty}\frac{\mathbb{E}[R^{\beta,l}(\nladder)]}{c_{\beta}}=\lim_{\beta \to +\infty}\int_0^{+\infty}\P(R^{\beta,l}(\nladder)>c_{\beta}s)ds.
$$
By Lemma \ref{prop7}, for any $\beta > \beta_{\gamma}$,
$$
\int_0^{+\infty}\P(R^{\beta,l}(\nladder)>c_{\beta}s)ds \leq \sum_{n=0}^{+\infty}\P(R^{\beta,l}(\nladder)>c_{\beta}n) \leq \sum_{n=0}^{+\infty}\gamma^n <+\infty.
$$
Therefore, the Dominated Convergence Theorem allow us to put the limit inside the integral as follows
$$
\lim_{\beta \to +\infty}\int_0^{+\infty}\P(R^{\beta,l}(\nladder)>c_{\beta}s)ds=\int_0^{+\infty}\lim_{\beta \to +\infty}\P(R^{\beta,l}(\nladder)>c_{\beta}s)ds = \int_0^{+\infty}e^{-s}ds=1.
$$
This and Proposition \ref{explim2} imply \eqref{p1t3}.

For any $\beta\geq 0$, for any $u \in \pcons$ and for any $s>0$,
$$
\P(R^{\beta, u}(\ncons)\geq c_{\beta}s)=
\P(R^{\beta, u}(\ncons)\geq c_{\beta}s , E_{\beta,u})+\P(R^{\beta, u}(\ncons)\geq c_{\beta}s, E_{\beta,u}^c),
$$
where
$$
E_{\beta,u}=
\{ R^{\beta, u}(\pladder)<\min\{1, R^{\beta, u}(\ncons)\},
R^{\beta, u}(\nladder)<R^{\beta, u}(\ncons)+1 \}.
$$
Note that the event $E_{\beta,u}$ implies that
$$
|R^{\beta, u}(\ncons)-(\inf\{t>R^{\beta, u}(\pladder): U_t^{\beta,u}\in\nladder\} - R^{\beta, u}(\pladder))|\leq 1.
$$
Therefore, by Markov property and by the symmetric properties of the process, for any $l \in \pladder$, we have that 
\begin{equation} \label{eqt3}
\P\left(\frac{R^{\beta, l}(\nladder)}{c_{\beta}}\geq s+\frac{1}{c_{\beta}}\right)\P\left(E_{\beta,u}\right)\leq \P(R^{\beta, u}(\ncons)\geq c_{\beta}s , E_{\beta,u}) \leq \P\left(\frac{R^{\beta, l}(\nladder)}{c_{\beta}}\geq s-\frac{1}{c_{\beta}}\right)\P\left(E_{\beta,u}\right).
\end{equation}
By Theorem \ref{fastconsensus} and Remark \ref{remarkteo2}, for any $u \in \pcons$,
$$
\lim_{\beta \to +\infty}\P(E_{\beta,u})=1,
$$
and then,
$$
\lim_{\beta \to +\infty}\P(R^{\beta, u}(\ncons)\geq c_{\beta}s , E_{\beta,u})=\lim_{\beta \to +\infty}\P(R^{\beta, u}(\ncons)\geq c_{\beta}s).
$$


By Proposition \ref{cbetabound},
$
\displaystyle\lim_{\beta \to +\infty}c_{\beta}^{-1}=0.
$
Therefore, Remark \ref{remark1} implies that 
$$
\lim_{\beta \to +\infty} \P\left(\frac{R^{\beta, l}(\nladder)}{c_{\beta}}\geq s+\frac{1}{c_{\beta}}\right)=\lim_{\beta \to +\infty}\P\left(\frac{R^{\beta, l}(\nladder)}{c_{\beta}}\geq s-\frac{1}{c_{\beta}}\right)=e^{-s}.
$$
Putting this together with \eqref{eqt3}, we conclude that for any $u \in \pcons$,
$$
\frac{R^{\beta,u}(\ncons)}{c_{\beta}}\to \text{Exp}(1) \text{ in distribution, as } \beta \to +\infty.
$$
To finish the proof, just note that the Dominated Convergence Theorem allow us to replace $c_{\beta}$ by $\mathbb{E}[R^{\beta, u}(\ncons)]$ as we did to prove that Equation \eqref{p1t3} holds.

\end{proof}

\begin{remark} \label{remark}
Since for any $u \in \pcons$
$$
\lim_{\beta \to +\infty}\frac{\mathbb{E}[R^{\beta, u}(\ncons)]}{c_{\beta}}=1,
$$
Proposition \ref{cbetabound} implies that 
$$
\lim_{\beta \to +\infty}\frac{\mathbb{E}[R^{\beta, u}(\ncons)]}{C_1 e^{\beta}}\geq 1.
$$
In particular, this implies that $\displaystyle \lim_{\beta \to \infty}\mathbb{E}[R^{\beta, u}(\ncons)]=+\infty$.

\end{remark}

\section{Proof of Theorem \ref{fastconsensus2}}
\label{sec:robot}




The proof of Theorem \ref{fastconsensus2} uses the same strategy already used to prove Theorem \ref{fastconsensus}.  
Let $(\hat{T}_n: n\geq 1)$ be the jumping times of the process $(\hat{U}_t^{\beta,u})_{t\in [0,+\infty)}$
and let $((\hat{A}_n,\hat{O}_n):n\geq 1)$ be the sequence of pairs associated to them. As for the case without robot, $\hat{O}_n$ is the opinion expressed at time $\hat{T}_n$, but $\hat{A}_n$ can be either an actor belonging to $\A$ or the robot, i.e., $\hat{A}_n \in \A \cup \{*\}$. 

For any initial list $u \in \str$ and for any $n \geq 1$, we define the event
$$
\hat{M}_n^u \mydef
\begin{cases}
H_n\left(\hat{U}_{\hat{T}_{n-1}}^{\beta,u}\right) &\text{, if }  \max\{|\hat{U}_{\hat{T}_{n-1}}^{\beta,u}(a)|: a \in \A\} > \strength, \\
H_n\left(\hat{U}_{\hat{T}_{n-1}}^{\beta,u}\right)\cup \{\hat{A}_n=*, \hat{O}_n=+1\} &\text{, if }  \max\{|\hat{U}_{\hat{T}_{n-1}}^{\beta,u}(a)|: a \in \A\} = \strength, \\
\{\hat{A}_n=*, \hat{O}_n=+1\} &\text{, if }  \max\{|\hat{U}_{\hat{T}_{n-1}}^{\beta,u}(a)|: a \in \A\} < \strength, 
\end{cases}
$$
where
$$
H_n\left(\hat{U}_{\hat{T}_{n-1}}^{\beta,u}\right) \mydef \left\{\hat{A}_n \in \text{argmax}\{|\hat{U}_{\hat{T}_{n-1}}^{\beta,u}(a)|: a \in \A\} \text{\ and \ } \hat{O}_n\hat{U}_{\hat{T}_{n-1}}^{\beta,u}(\hat{A}_n) \geq 0\right\}. 
$$
$\hat{M}_n$ is the event in which the most likely choice of the pair $(\hat{A_n}, \hat{O_n})$ occurs.

For any $\beta \geq 0$ and for any $\strength >0$, the invariant probability measure of the embedded process $(\hat{U}^{\beta,u}_{\hat{T}_n})_{n \geq 0}$ will be denoted $\tilde{\nu}^{\beta, \strength}$.
For any initial list $u \in \str$, let
$$
\hat{\tau}(u)\mydef \inf\{n\geq1: \hat{A}_n \in \{\hat{A}_1,..., \hat{A}_{n-1}\} \cup\{a \in \A: u(a)=0\} \cup \{*\}\}.
$$

Proposition \ref{teo:m1m2} is a version of Proposition \ref{teo:m1m} for the model with a robot.

\begin{proposition} \label{teo:m1m2}
\begin{enumerate}
\item[]
    \item For any $\beta \geq 0$, for any $\strength >0$, for any initial list $u \in \str$ and for any $m \geq 1$,
$$
\P\left(\bigcap_{j=1}^m \hat{M}_j^u\right) \geq (\hat{\zeta}_{\beta})^m,
$$
where
$$
\hat{\zeta}_{\beta} \mydef \frac{e^{\beta (1-\gamma)}}{e^{\beta (1-\gamma)}+e^{-\beta (1-\gamma)}+2N},
$$
with $\gamma=\max\{\lceil \strength\rceil - \strength, \strength-\lfloor \strength\rfloor\}$.
\


\item 
If $\strength > N-1$, 
for any $u \in \str$,
$$
\P\left(\hat{U}^{\beta,u}_{\hat{T}_{2\lfloor \strength \rfloor+3N+2}} \in \pladderr\ \Big| \bigcap_{j=1}^{2\lfloor \strength \rfloor+3N+2}\hat{M}_j^u \right)=1.
$$

\end{enumerate}
\end{proposition}

\begin{proof}
The proof of  Part 1 of Proposition \ref{teo:m1m2} can be done exactly as we did to prove Part 1 of Proposition \ref{teo:m1m} in Section \ref{sec:auxiliary}.

To prove Part 2 of Proposition \ref{teo:m1m2} we first 
observe that $\hat{\tau}(u) \leq N+1$, by definition.
Therefore, the Markov property implies the following inequality
$$
\P\left(\hat{U}^{\beta,u}_{\hat{T}_{2\lfloor \strength \rfloor+3N+2}} \in \pladderr\ \Big| \bigcap_{j=1}^{2\lfloor \strength \rfloor+3N+2}\hat{M}_j^u \right)
\geq
\P\left(\bigcap_{a \in \A}\{\hat{U}_{\hat{T}_{\hat{\tau}(u)}}^{\beta,u}(a) > -\strength  \} \ \Big| \ \displaystyle\bigcap_{j=1}^{\hat{\tau}(u)}\hat{M}_j^u \right) \times
$$
$$
\P\left(\hat{U}_{\hat{T}_{3\lfloor \strength \rfloor+2N+2}}^{\beta,u} \in \pladderr \ \Big| \ \displaystyle\bigcap_{j=\hat{\tau}(u)+1}^{3\lfloor \strength \rfloor+2N+2}\hat{M}_j^u \cap \bigcap_{a \in \A}\{\hat{U}_{\hat{T}_{\hat{\tau}(u)}}^{\beta,u}(a) > -\strength  \}\right).
$$
We will show that the two terms on the right-hand side of the equation above are equal to $1$.

Note that
$$
\displaystyle\bigcap_{j=1}^{\hat{\tau}(u)}\hat{M}_j^u \cap \{\hat{A}_{\hat{\tau}(u)} \neq * \} \subset \{|\hat{U}^{\beta,u}_{T_{\hat{\tau}(u)-1}}(\hat{A}_{\hat{\tau}(u)})| \geq \strength\}.
$$
At the same time, following the arguments of Lemma \ref{lemma1m1m}, we conclude that
\begin{equation} \label{eqproprobot}
\displaystyle\bigcap_{j=1}^{\hat{\tau}(u)}\hat{M}_j^u \cap \{\hat{A}_{\hat{\tau}(u)} \neq * \} \subset \{|\hat{U}^{\beta,u}_{T_{\hat{\tau}(u)-1}}(\hat{A}_{\hat{\tau}(u)})| \leq N-1\}.
\end{equation}
Therefore, if $\strength > N-1$, then 
$$
\displaystyle\bigcap_{j=1}^{\hat{\tau}(u)}\hat{M}_j^u \cap \{\hat{A}_{\hat{\tau}(u)} \neq * \}\subset\{|\hat{U}^{\beta,u}_{T_{\hat{\tau}(u)-1}}(\hat{A}_{\hat{\tau}(u)})| \leq N-1\}\cap \{|\hat{U}^{\beta,u}_{T_{\hat{\tau}(u)-1}}(\hat{A}_{\hat{\tau}(u)})| \geq \strength\}=\emptyset.
$$ As a consequence, 
$$\displaystyle\bigcap_{j=1}^{\hat{\tau}(u)}\hat{M}_j^u = \displaystyle\bigcap_{j=1}^{\hat{\tau}(u)}\hat{M}_j^u \cap \{\hat{A}_{\hat{\tau}(u)} = * \}.$$




Therefore,
$$
\P\left(\bigcap_{a \in \A}\{\hat{U}_{\hat{T}_{\hat{\tau}(u)}}^{\beta,u}(a) > -\strength  \} \ \Big| \ \displaystyle\bigcap_{j=1}^{\hat{\tau}(u)}\hat{M}_j^u \right)=
$$
$$
\P\left(\bigcap_{a \in \A}\{\hat{U}_{\hat{T}_{\hat{\tau}(u)}}^{\beta,u}(a) > -\strength  \} \ \Big| \ \displaystyle\bigcap_{j=1}^{\hat{\tau}(u)}\hat{M}_j^u \cap \{\hat{A}_{\hat{\tau}(u)} = * \} \right).
$$
By definition, 
$$
\bigcap_{j=1}^{\hat{\tau}(u)}\hat{M}_j^u \cap \{\hat{A}_{\hat{\tau}(u)} = * \} \subset \bigcap_{a \in \A}\{-\strength \leq \hat{U}_{\hat{T}_{\hat{\tau}(u)-1}}^{\beta,u}(a) \leq \strength\}.
$$
Putting it together with the fact that 
$\hat{A}_{\hat{\tau}(u)} = * $ implies that $\hat{O}_{\hat{\tau}(u)} = +1$, we conclude that 
$$
\P\left(\bigcap_{a \in \A}\{\hat{U}_{\hat{T}_{\hat{\tau}(u)}}^{\beta,u}(a) > -\strength  \} \ \Big| \ \displaystyle\bigcap_{j=1}^{\hat{\tau}(u)}\hat{M}_j^u \cap \{\hat{A}_{\hat{\tau}(u)} = * \} \right)=1.
$$


With arguments similar to those used to prove Proposition \ref{teo:m1m}, we prove that
$$
\P\left(\hat{U}_{\hat{T}_{\hat{\tau}(u)+2\lfloor \strength \rfloor+2N+1}}^{\beta,u} \in \pladderr \ \Big| \ \displaystyle\bigcap_{j=1}^{2\lfloor \strength \rfloor+2N+1}\hat{M}_j^u \cap \bigcap_{a \in \A}\{\hat{U}_{\hat{T}_{\hat{\tau}(u)}}^{\beta,u}(a) > -\strength  \}\right)=1.
$$
This concludes the proof of Part 2 of Proposition \ref{teo:m1m2}.

\end{proof}

Now, we prove Theorem \ref{fastconsensus2}.

\begin{proof}

Using the same strategy used to prove Theorem \ref{fastconsensus}, we first observe that
$$
\nu^{\beta, \strength}(\pladderr)
\geq
\frac{1}{1+ \displaystyle\frac{
(2N+1)e^{\beta (\lfloor \strength \rfloor +1)}}{\tilde{\nu}^{\beta, \strength}(\pladderr)}\sum_{u' \not\in \pladderr}\frac{\tilde{\nu}^{\beta, \strength}(u')}{\hat{q}_{\beta}(u')}},
$$
where for any $u \in \str$,
$$
\hat{q}_{\beta}(u) \mydef \sum_{a \in \A}(e^{\beta u(a)}+e^{-\beta u(a)})+ e^{\beta \strength}.
$$
If $\max\{|u(a)|: a \in \A\} \geq \lfloor \strength \rfloor +2$, then $\hat{q}_{\beta}(u) \geq e^{\beta (\lfloor \strength \rfloor +2)}$, and therefore,
$$
\frac{\tilde{\nu}^{\beta, \strength}(u)}{\hat{q}_{\beta}(u)} \leq \tilde{\nu}^{\beta, \strength}(u)e^{-\beta (\lfloor \strength \rfloor +2)}.
$$

In the case $u \not\in \pladderr$ and $\max\{|u(a)|: a \in \A\} <\lfloor \strength \rfloor +2$, using Proposition \ref{teo:m1m2} we can prove that there exists $\hat{C}'=\hat{C}'(N)>0$ such that $\tilde{\nu}^{\beta, \strength}(u) \leq \hat{C}' e^{-\beta \strength}$ exactly as we did to prove Proposition \ref{decay} in Section \ref{sec: fastconsensus}.
In this case, it follows that
$$
\frac{\tilde{\nu}^{\beta, \strength}(u)}{\hat{q}_{\beta}(u)} \leq \hat{C}' e^{-\beta (\lfloor \strength \rfloor +2)}.
$$
From this point on, the  proof of Part 2.1 of Theorem \ref{fastconsensus2} can be done exactly as we did to prove Part 1 of Theorem \ref{fastconsensus}.

Considering Proposition \ref{teo:m1m2}, the proof of Part 2.2 of Theorem \ref{fastconsensus2} can be done similarly as we did
to prove Part 2 of 
Theorem \ref{fastconsensus}.

\end{proof}

\section{Discussion}

Our model describes the time evolution of a ``social bubble''. By this we mean a group of highly connected social actors and disconnected from the actors outside the group. Theorems \ref{fastconsensus}, \ref{metastable} and \ref{fastconsensus2} rigorously analyze the behavior of such a system of interacting social actors in a highly polarized situation.

One of the main features of our model is the disposition of each actor to always take into account the opinion of the other actors in the network. We conjecture that this is a typical feature of a social network in which all the actors have similar psychographic profiles. In a highly polarized situation, this feature makes our model a  kind of ``consensus building machine''.

The high polarization in our model of social network is compatible with the ``filter bubble'' hypothesis introduced by \cite{pariser}. In a nutshell, this hypothesis links network polarization with algorithmic filtering of ideas and opinions. This creates ``echo chambers'', amplifying opinions within the network \cite{bubble}.

The Brazilian 2018 and 2022 elections served as an inspiration for our model and suggested the following question: is social media campaigning enough to change in a quite short period of time the voting intention of a significant portion of voters?

The results described by Theorems \ref{fastconsensus}, \ref{metastable} and \ref{fastconsensus2} seem to reproduce qualitatively the fast modification of the voting intentions observed before the first rounds of the 2018 and 2022 presidential elections in Brazil. 
However, a more detailed statistical analysis must be done to support the conjecture that our model offers an adequate explanation for the recent Brazilian electoral dynamics.

\section*{Acknowledgments}
This work is part of USP project {\em Mathematics, computation, language and the brain} and FAPESP project {\em Research, Innovation and Dissemination Center for Neuromathematics} (grant 2013/07699-0).  AG was partially supported by CNPq fellowship (grant 314836/2021-7). KL was successively supported  by Capes fellowship (grant 88887.340896/2019-00), CNPq  fellowship (grant 164080/2021-0) and FAPESP fellowship (grants 2022/07386-0 and 2023/12335-9).

We thank M. Cassandro, M. Galves, J. A. Peschanski, R. Prôa and F. Penafiel for many interesting discussions. Finally, we thank two anonymous reviewers for their comments and suggestions.

\bibliographystyle{apalike}
\bibliography{bibliografia}

\vspace{1cm}

\noindent
Antonio Galves \\ Deceased September 5, 2023 \\ Instituto de Matemática e Estatística \\ Universidade de São Paulo \\ Rua do Matão, 1010 \\ São Paulo-SP, 05508-090 \\ Brazil

\vspace{0.5cm}

\noindent
Kádmo de Souza Laxa \\ Faculdade de Filosofia, Ciências e Letras de Ribeirão Preto  \\ Universidade de São Paulo \\
Av. Bandeirantes, 3900\\
Ribeirão Preto-SP, 14040-901 \\
Brazil \\
e-mail address: \texttt{kadmo.laxa@usp.br}

\end{document}